\documentclass[11pt,oneside,english,jou]{amsart}
\usepackage[T1]{fontenc}
\usepackage[latin9]{inputenc}
\usepackage{babel}
\usepackage{verbatim}
\usepackage{mathrsfs} 	
\usepackage{amstext}
\usepackage{amsthm}
\usepackage{amssymb}
\usepackage{amsfonts}
\usepackage{amsmath}
\usepackage{esint}
\usepackage{enumerate}
\usepackage[unicode=true,pdfusetitle,
bookmarks=true,bookmarksnumbered=false,bookmarksopen=false,
breaklinks=false,pdfborder={0 0 1},colorlinks=false]
{hyperref}
\usepackage[margin=2.5cm]{geometry}
\usepackage[foot]{amsaddr}
\usepackage{mathtools}
\usepackage{ dsfont }
\usepackage[shortlabels]{enumitem}
\usepackage{bm}

\makeatletter
\numberwithin{equation}{section}
\numberwithin{figure}{section}
\theoremstyle{plain}
\newtheorem{thm}{\protect\theoremname}[section]
\theoremstyle{definition}
\newtheorem{rem}[thm]{\protect\remarkname}
\theoremstyle{definition}
\newtheorem{defn}[thm]{\protect\definitionname}
\theoremstyle{plain}
\newtheorem{prop}[thm]{\protect\propositionname}
\theoremstyle{plain}
\newtheorem{lem}[thm]{\protect\lemmaname}
\theoremstyle{plain}

\theoremstyle{plain}

\theoremstyle{definition}

\theoremstyle{definition}

\theoremstyle{definition}
\newtheorem{examplex}[thm]{\protect\examplename}
\theoremstyle{definition}

\newenvironment{example}
{\pushQED{\qed}\examplex}
{\popQED\endexamplex}

\usepackage{tikz}
\usetikzlibrary{shapes,arrows}
\usepackage{verbatim}
\usepackage{amsthm}
\usepackage{amstext}

\DeclareMathOperator{\diam}{diam}

\DeclareMathOperator{\supp}{supp}

\DeclareMathOperator{\Int}{Int}

\newcommand{\R}{\mathbb R}

\newcommand{\N}{\mathbb N}
\newcommand{\Q}{\mathbb Q}

\newcommand{\MM}{\mathcal M}

\renewcommand{\b}{\color{blue}}

\newcommand{\eps}{\varepsilon}

\newcommand{\mA}{\mathcal{A}}

\newcommand{\hdim}{\dim_H}
\newcommand{\adim}{\dim_A}
\newcommand{\lhdim}{\underline{\dim}_H}
\newcommand{\uhdim}{\overline{\dim}_H}

\newcommand{\ld}{\underline{d}}
\newcommand{\ud}{\overline{d}}

\DeclareMathOperator{\Lip}{Lip}

\DeclareMathOperator{\Gr}{Gr}

\makeatother

\providecommand{\conjecturename}{Conjecture}
\providecommand{\corollaryname}{Corollary}
\providecommand{\definitionname}{Definition}
\providecommand{\examplename}{Example}
\providecommand{\lemmaname}{Lemma}
\providecommand{\problemname}{Problem}
\providecommand{\propositionname}{Proposition}
\providecommand{\remarkname}{Remark}
\providecommand{\theoremname}{Theorem}
\providecommand{\taskname}{Task}

\def\diam{{\rm diam}}

\def\supp{{\rm supp}}

\newcommand{\lam}{\lambda}

\def\Lam{\Lambda}

\newcommand{\om}{\omega}

\def\Om{\Omega}

\newcommand{\sig}{\sigma}

\def\N{{\mathbb N}}

\def\Ak{{\mathcal A}}
\def\Bk{{\mathcal B}}

\def\Mk{{\mathcal M}}

\def\Vk{{\mathcal V}}

\def\be{\begin{equation}}
	\def\ee{\end{equation}}

\newcommand{\Ek}{{\mathcal E}}
\newcommand{\Fk}{{\mathcal F}}

\def\fin{{\mathrm{fin}}}
\def\essinf{{\mathrm{ess inf}}}
\def\esssup{{\mathrm{ess sup}}}

\begin{document}
\title{Universal projection theorems with applications to multifractal analysis and the dimension of every ergodic measure on self-conformal sets simultaneously}

\author{Bal\'azs B\'ar\'any$^1$}
\address{$^1$Department of Stochastics,	Institute of Mathematics, Budapest University of Technology and Economics, M\H{u}egyetem rkp. 3, H-1111 Budapest, Hungary}
\email{balubsheep@gmail.com}

\author{K\'aroly Simon$^{1,2}$}
\address{$^2$HUN-REN--BME Stochastics Research Group, Budapest University of Technology and Economics, M\H{u}egyetem rkp. 3, H-1111 Budapest, Hungary}
\email{simonk@math.bme.hu}

\author{Adam \'Spiewak$^3$}
\address{$^3$ Institute of Mathematics of the Polish Academy of Sciences, ul. \'Sniadeckich 8, 00-656 Warszawa, Poland}
\email{ad.spiewak@gmail.com}

\thanks{B. B\'ar\'any acknowledges support from the grant NKFI FK134251. B. B\'ar\'any and K. Simon were supported by the grants NKFI K142169 and KKP144059 "Fractal geometry and applications". A. \'Spiewak was partially supported by the National Science Centre (Poland) grant 2020/39/B/ST1/02329. We are grateful to Thomas Jordan and Boris Solomyak for useful discussions which have inspired this work.}

\date{\today}

\begin{abstract}
We prove a \textit{universal projection theorem}, giving conditions on a parametrized family of maps $\Pi_\lam : X \to \R^d$ and a collection $\Mk$ of measures on $X$ under which for almost every $\lam$ equality $\hdim \Pi_\lam \mu = \min\{d, \hdim \mu\}$  holds for all measures $\mu \in \Mk$ \textbf{simultaneously} (i.e. on a full measure set of $\lam$'s independent of $\mu$). We require family $\Pi_\lam$ to satisfy a transversality condition and collection $\Mk$ to satisfy a new condition called \textit{relative dimension separability}. Under the same assumptions, we also prove that if the Assouad dimension of $X$ is smaller than $d$, then for almost every $\lam$, projection $\Pi_\lam$ is nearly bi-Lipschitz (i.e. with pointwise $\alpha$-H\"older inverse for every $\alpha \in (0,1)$) at $\mu$-a.e. $x$, for all $\mu \in \Mk$ simultaneously. Our setting encompasses families of orthogonal projections, natural projections corresponding to conformal iterated function systems, and non-autonomous or random IFS.

As applications, we provide novel results on the multifractal analysis, giving formula for the Hausdorff dimension of a level set of the local dimension for a typical (w.r.t the translation parameter) self-similar measure on the line, valid for the full range spectrum (including the decreasing part of the spectrum; previous results were covering only the increasing part). 

Among another applications, we prove that given a parametrized contracting conformal IFS satisfying the transversality condition, for almost every parameter the dimension formula holds for all ergodic shift-invariant measures simultaneously. We also prove that the dimension part of the Marstrand's projection theorem holds simultaneously for the collection of all ergodic measures on a strongly separated self-conformal set and for the collection of all Gibbs measures on a self-conformal set (without any separation).

\end{abstract}

\keywords{iterated function systems, transversality, multifractal analysis, orthogonal projections, dimension theory}

\subjclass[2020]{37E05 (Dynamical systems involving maps of the interval (piecewise continuous, continuous, smooth)), 28A80 (Fractals), 28A75 (Length, area, volume, other geometric measure theory)}

\maketitle

\section{Introduction}

\subsection{Marstrand's projection theorem} Determining the dimension of certain objects, like sets and measures, plays a crucial role in geometric measure theory and fractal geometry, just like understanding how different actions change the value of the dimension. The classical \textit{projection theorem}, originated to Marstrand \cite{Marstrand} and later generalized in several ways e.g. by Falconer \cite{Falconerproj}, Kaufman \cite{Kaufman2}, Mattila \cite{Mattila75}, claims that the dimension of the orthogonal projection of a set in a typical direction (in some proper sense) does not drop with respect to the natural upper bound (the minimum of the dimension of the space, where we project, and the dimension of the projected set). For the present work, the most relevant is a version for measures \cite{HuTaylorProjections, HuntKaloshin97, SauerYorke97}, which states that for every finite Borel measure $\mu$ on $\R^n$
\begin{equation}\label{eq: marstrand} \hdim(P_V \mu) = \min\{ d,  \hdim \mu \}  \text{ for } \gamma \text{-a.e. } V \in \Gr(d,n),
\end{equation}
where $\hdim$ denotes the Hausdorff dimension\footnote{in fact, for general measures one should consider the upper and lower Hausdorff dimensions separately. For simplicity, we omit this subtlety in the introduction.}, $\Gr(d,n)$ is the Grassmannian manifold of $d$-dimensional linear subspaces in $\R^n$ endowed with the unique $O(n)$-invariant probability measure $\gamma$ (see \cite[Section 3.9]{mattila}) and $P_V$ denotes the orthogonal projection onto $V \in \Gr(d,n)$ - see Section \ref{sec: prelim} for definitions.

Marstrand's projection theorem has been generalized and extended in several directions, with the methods behind its proof becoming an underlying paradigm for many areas of research in geometric measure theory \cite{mattila, MattilaFourier, SixtyProjections}, fractal geometry \cite{PollicottSimonDeleted, SolAC, SolomyakTransversSurvey} or dynamical systems \cite{SYC91, Tsujii}. The method is often referred to as the \textit{transversality technique}. Due to its wide range of applications, it has been extended to generalized projection schemes, which apply to several settings at once, see e.g. \cite{SolFamilies, PeresSchlag, LPS02, BongersTaylor23, BSS}. The aim of this paper is to improve the strength of this technique in the setting of generalized projections and utilize it for novel applications in geometric measure theory and fractal geometry.

\subsection{Main results - universal projection theorem} The main result of this paper is Theorem \ref{thm: general proj}. Let us explain it first in a non-technical manner. Let $\Pi_\lam : X \to \R^d$ be a collection of Lipschitz maps (\textit{projections}) parametrized by a parameter $\lam \in U$ and satisfying the transversality condition with respect to a measure $\eta$ on $U$ (see Section \ref{sec: main} for precise definitions and assumptions). A direct extension of \eqref{eq: marstrand} to such a general setting states that given a finite measure $\mu$ on $X$ the following holds
\begin{equation}\label{eq: proj measure intro} \hdim \Pi_\lam \mu = \min\{ d, \hdim \mu\}  \text{ for } \eta\text{-a.e. } \lam \in U,
\end{equation}
see e.g. \cite[Theorem 6.6.2]{BSS}. The first goal of this paper is to study conditions under which \eqref{eq: marstrand}, and more generally \eqref{eq: proj measure intro}, hold for \textit{all} measures in a given class of measures $\Mk$ \textit{simultaneously}, i.e. with a measure zero set of exceptional projections independent of measure $\mu \in \Mk$. More precisely, we aim at finding conditions on $\Mk$ guaranteeing that
\begin{equation}\label{eq: sim general} \eta \left( \left\{ \lam \in U : \underset{\mu \in \Mk}{\exists} \hdim \Pi_\lam \mu \neq \min\{ d, \hdim \mu\} \right\} \right) = 0
\end{equation}
holds. Clearly, this question is relevant only for uncountable families $\Mk$ and it is easy to find examples of collections $\Mk$ for which \eqref{eq: sim general} does not hold (see Example \ref{eq: orthogonal proj non example} for the case of orthogonal projections). Our main result - Theorem \ref{thm: general proj} - identifies a condition on $\Mk$, which we call \textit{relative dimension separability} (see Definition \ref{defn: relative dimension separability}), guaranteeing that \eqref{eq: sim general} holds. We refer to the result as the \textit{universal projection theorem}\footnote{the name \textit{universal projection theorem} is inspired by seminal results and constructions known in information theory as \textit{universal source codings}, providing compression algorithms operating in an optimal rate for any source distribution in a given class, without any prior knowledge of this distribution, see e.g. \cite[Chapter 13]{CoverElements}. A famous example is the Lempel-Ziv coding \cite{LZ77, LZ78} achieving optimal rate for any stationary ergodic source distribution \cite[Theorem 13.5.3]{CoverElements}.}.  We shall emphasize, that the projection scheme studied in this paper does not require any extra assumptions than the classical ones, hence the simultaneous result \eqref{eq: sim general} holds essentially whenever the standard transversality technique can be applied.

We also study the same problem in the context of \textit{embeddings}, corresponding to the injectivity properties of the projections. This is an important topic in analysis and geometric measure theory, studied by numerous authors \cite{Mane81, SYC91, EFNT94, BAEFN93, HK99, olson2002bouligand, RossiShmerkinHolderCoverings}. In recent years there has been an increasing interest in almost surely injective embeddings of measures, due to its applications to compression of analog signals \cite{WV10, LosslessAnalogCompression, mmdimcompress} and time-delayed embeddings of dynamical systems \cite{SSOY98, BGSTakens, BGS22}. Of particular importance is establishing guarantees on the regularity of the embedding \cite{HK99, olson2002bouligand, baranski2023regularity, SRegularity}. In the fractal setting, strongest regularity properties have been obtained in terms of the Assouad dimension \cite{olson2002bouligand} (see also \cite{Rob11}), guaranteeing that the inverse to the orthogonal projection of a compact set $X \subset \R^N$ onto a typical $k$-dimensional plane is almost bi-Lipschitz (in particular: $\alpha$-H\"older for every $\alpha \in (0,1)$) provided that $\adim (X-X) < k$ (here $\adim$ denotes the Assouad dimension, see Definition \ref{defn:adim}). It was shown in \cite[Theorem 1.7.(iii)]{baranski2023regularity}, that given a probability measure $\mu$ on $X$, a weaker condition $\adim X < k$ suffices to guarantee that the inverse is defined $\mu$-almost everywhere and it is pointwise almost bi-Lipschitz at $\mu$-a.e. $x \in X$. In Theorem \ref{thm: general proj} we improve this result to a universal version for general projection schemes. In a spirit similar to \eqref{eq: sim general}, we prove in Theorem \ref{thm: general proj} that the following holds for relative dimension separable collections of measures $\Mk$ on $X$
\begin{equation}\label{eq: assouad general} \text{if } \adim X < d, \text{ then } \eta \left( \left\{ \lam \in U : \underset{\mu \in \Mk}{\exists}\  \Pi_\lam \text{ is not } \mu\text{-nearly bi-Lipschitz} \right\} \right) = 0,
\end{equation}
with the precise definition of the $\mu$-nearly bi-Lipschitz property given in Definition \ref{defn: nearly biLip}. Establishing the universal nearly bi-Lipschitz property \eqref{eq: assouad general} is a crucial step for further applications. As we explain later, in the context of iterated function systems, the nearly bi-Lipschitz property is equivalent to a new and non-trivial separation condition called \textit{exponential distance from the enemy} (EDE) (see Definition \ref{defn: ede} and Proposition \ref{prop: ede holder}). Its usefulness is verified by novel applications to the multifractal analysis, elaborated on below. See also Remark \ref{rem: ede useful}.

\subsection{Applications I - multifractal analysis of self-similar measures}

Our most important application of the universal projection theorem described above is to the multifractal analysis of self-similar measures. To explain it, let us first introduce the general setting of iterated functions systems.

Let $\Ak$ be a finite collection of indices and let $\Fk=\{f_i\colon\R^d\to\R^d\}_{i\in\Ak}$ be a finite collection of (strictly) contracting maps, called iterated function system (IFS), and let $p = (p_i)_{i\in\Ak}$ be a probability measure on $\Ak$. Then there exists a unique non-empty compact set $\Lambda$ and a unique compactly supported probability measure $\nu$ such that $\Lam$ is invariant with respect to $\Fk$ in the sense that $\Lambda=\bigcup \limits_{i\in\Ak}f_i(\Lambda)$ and $\nu$ is stationary, i.e. $\nu=\sum_{i\in\Ak}p_i(f_i)_*\nu$ - see Hutchinson~\cite{Hutchinson}. Measure $\nu$ is called \textit{self-similar} if $\Fk$ consists of similarity maps. The \textit{multifractal analysis} studies the level-set function $\alpha \mapsto \dim_H \left( \left\{x:d(\nu, x)=\alpha\right\} \right)$. It can be seen as a fine-scale analysis of the geometry of measure $\mu$ and originates from the study of the distribution of singularities on strange attractors in a chaotic, nonlinear dynamics \cite{HJKPSMultifractal}. Later, a rigorous mathematical approach has been developed and became one of the central topics in fractal geometry, see e.g. \cite{PesinBook} for a more detailed discussion. A basic heuristics for a self-similar measure corresponding to a probability vector $p$ and an IFS $\Fk=\{f_i(x)=\lambda_ix+t_i\}_{i\in\Ak}$ is that one expects the formula
\begin{equation}\label{eq: multifractal main} \dim_H \left( \left\{x:d(\nu, x)=\alpha\right\} \right) = \inf_{q\in\R}(\alpha q+T(q)),
\end{equation}
to hold, where for $q\in\R$ we define $T(q)$ to be the unique real number satisfying $\sum \limits_{i\in\Ak}p_i^q|\lambda_i|^{T(q)}=1$ (the quantity $T(q)$ is related to the $L^q$-spectrum of the self-similar measure, see Remark \ref{rem: L^q}). The maximal range of $\alpha$'s on which one can expect \eqref{eq: multifractal main} to hold is $\left[\min\limits_{i\in\Ak}\frac{\log p_i}{\log|\lambda_i|},\max\limits_{i\in\Ak}\frac{\log p_i}{\log|\lambda_i|}\right]$. Arbeiter and Patschke \cite{AP} showed that \eqref{eq: multifractal main} holds for the full range of $\alpha$'s if $\Fk$ satisfies the Open Set Condition. With the aid of the universal projection theorem we can prove the following result, addressing the overlapping case for typical parameters.

\begin{thm}\label{thm:multi intro}
	Fix $\lam_i \in (-1,1) \setminus \{0\}, i\in \Ak$ for a finite set $\Ak$. For each $t = (t_i)_{i \in \Ak} \in \R^{\Ak}$ and a probability vector $p = (p_i)_{i \in \Ak}$, let $\nu_{t,p}$ be the self-similar measure corresponding to the IFS $\Fk_t=\{f_i(x)=\lambda_ix+t_i\}_{i\in\Ak}$ on $\R$ and the probability vector $p$. The following holds for Lebesgue almost every $t \in \R^\Ak$. If the similarity dimension $s_0 = s(\Fk_t)$ of $\Fk_t$ satisfies $s_0 < 1$, then equality
	\[\dim_H \left( \left\{x:d(\nu_{t,p}, x)=\alpha\right\} \right)=\inf_{q\in\R}(\alpha q+T(q))\]
	holds for every $\alpha\in\left[\min\limits_{ i \in \Ak}\frac{\log p_i}{\log|\lambda_i|},\max\limits_{ i \in \Ak }\frac{\log p_i}{\log|\lambda_i|}\right]$ and every probability vector $p$ such that $p_i\neq |\lambda_i|^{s_0}$ for some $i \in \Ak$. 
\end{thm}

See also Theorem \ref{thm: multi high dim main} for the multidimensional version. Let us explain the significance of the above result. So far, the strongest result on the multifractal analysis of self-similar measures with overlaps is the one by Barral and Feng \cite[Theorem~1.2, Remark~7.3]{BarralFeng} (improving upon their earlier result \cite[Theorem~1.3, Theorem~6.4]{BarralFengselfaffine}), which establishes \eqref{eq: multifractal main} for self-similar measures with $s_0 < 1$ and satisfying the Exponential Separation Condition (recall that it is satisfied for $t \in \R^\Ak$ outside of a set of positive Hausdorff codimension, so in particular for almost every $t \in \R^\Ak$), for $\alpha\in\left[\min\limits_{i\in\Ak}\frac{\log p_i}{\log|\lambda_i|},\frac{\sum_i|\lambda_i|^{s_0}\log p_i}{\sum_i|\lambda_i|^{s_0}\log|\lambda_i|}\right]$. This range corresponds to the \textit{increasing} part of the spectrum. The remaining range $\left[ \frac{\sum_i|\lambda_i|^{s_0}\log p_i}{\sum_i|\lambda_i|^{s_0}\log|\lambda_i|}, \max\limits_{i\in\Ak}\frac{\log p_i}{\log|\lambda_i|} \right]$ (the \textit{decreasing} part of the spectrum), uncovered by their results, corresponds to $\alpha$'s for which the infimum in \eqref{eq: multifractal main} is attained for $q<0$ and hence is related to studying the $L^q$-spectrum of the self-similar measure for negative $q$. This is known to be notoriously difficult and there are no corresponding results available for typical self-similar measures with overlaps in this range. With the use of the nearly bi-Lipschitz property from Theorem \ref{thm: general proj} we are able to circumvent this problem. According to our best knowledge, Theorem \ref{thm:multi intro} is the first result on the multifractal spectrum of typical self-similar measures with overlaps, which covers the full spectrum. Let us emphasize that while the introduction of additive combinatorics methods of Hochman \cite{H} and Shmerkin \cite{Shmerkinlq} to the study of self-similar measures have led to deep breakthroughs in recent years (including results of Barral and Feng \cite{BarralFeng} on multifractal formalism for self-similar measures), these methods have not been yet able to provide results on the multifractal analysis in the full range $\left[\min\limits_{ i \in \Ak}\frac{\log p_i}{\log|\lambda_i|},\max\limits_{ i \in \Ak }\frac{\log p_i}{\log|\lambda_i|}\right]$. The progress in Theorem \ref{thm:multi intro} is based on the classical transversality technique, with establishing \eqref{eq: assouad general} as the crucial ingredient. It is easy to extend our results to the Birkhoff spectra of continuous potentials.

\subsection{Applications II - dimension of every ergodic measure on self-conformal sets simultaneously}

The transversality method was first applied to the case of iterated function systems by Pollicott and Simon \cite{PollicottSimonDeleted}, and later several generalizations have appeared, see for instance \cite{SolPeresBernoulli, SolFamilies, PeresSchlag, LPS02, SSUPacific,SSUTrans, BSSSTypical, BSSSExposition24}. See also \cite{SolomyakTransversSurvey} for a recent survey and \cite{BSS} for an in-depth discussion. Roughly speaking, the dimensional parts of these statements can be summarized as follows. Let $\Fk_\lam=\{f_i^{\lambda}\colon\R^d\to\R^d\}_{i\in\Ak},\ \lam \in U$ be a family of parametrized conformal ($C^{1+\theta}$) IFS, which satisfies the transversality condition with respect to a probability measure $\eta$ on $U$ (see Example \ref{ex:parametrized IFS} for the precise assumptions). Let $\Pi_\lambda : \Ak^\N \to \R^d,\ \Pi_\lam(\om) = \lim \limits_{n \to \infty} f_{\om_1} \circ \cdots \circ f_{\om_n}(x)$ be the natural projection map corresponding to $\Fk_\lam$ and let $\mu$ be an ergodic left-shift invariant measure on $\Ak^\N$. Then (see e.g. \cite[Theorem 14.4.2]{BSS})
\begin{equation}\label{eq:dimmeasurebasic}
	\dim_H( \Pi_\lambda \mu)=\min\left\{d,\frac{h(\mu)}{\chi(\mu,\lambda)}\right\} \text{ for } \eta\text{-a.e. } \lam,
\end{equation} where $h(\mu)$ denotes the entropy of $\mu$ and $\chi(\mu,\lambda)$ denotes the Lyapunov exponent of $\mu$ with respect to $\Fk_\lambda$. With the use of our methods, we can improve this classical result to a version holding simultaneously for all ergodic measures on $\Ak^\N$:
\begin{equation}\label{eq: simult IFS} \eta \left( \left\{ \lam \in U : \text{ there exists } \mu \in \Ek_\sigma(\Ak^\N) \text{ such that} \hdim (\Pi_\lam \mu) \neq  \min\left\{d, \frac{h(\mu)}{\chi(\mu, \lam)} \right\} \right\} \right) = 0,
\end{equation}
where $\Ek_\sigma(\Ak^\N)$ is the set of all shift-invariant ergodic Borel probability measures on $\Ak^\N$. See Theorem \ref{thm: main trans IFS} for the precise statement. 

The setting of our general result is wide enough to encompass also other cases,
like non-autonomous IFS, where in each iterate we choose different IFS with the same parametrization (see Nakajima~\cite{Nak}) or when in each iterate we add an independent, identically distributed error (see Jordan, Pollicott and Simon~\cite{JPS07}, Koivusalo~\cite{Koivusalo} and Liu and Wu \cite{LW}). More details are given in Section \ref{sec: further examples}.

A strong dimension result holding simultaneously for all ergodic measures on self-similar sets in $\R$ was obtained by Jordan and Rapaport \cite[Theorem 1.1]{JordanRapaport}. They showed that if an IFS consisting of similarities on $\R$ satisfies the Exponential Separation Condition (ESC), then \eqref{eq:dimmeasurebasic} holds for all ergodic measures $\mu$ on $\Ak^\N$. Our result \eqref{eq: simult IFS} extends this to typical non-linear IFS on $\R^n$ satisfying the transversality condition. Very recently, Rapaport \cite{RapaportAnalytic} extended considerably Hochman's result \cite{H}, proving that \eqref{eq:dimmeasurebasic} holds for any IFS on $\R$ consisting of analytic contractions satisfying ESC and every Bernoulli measure $\mu$ on $\Ak^\N$. It however seems a challenging problem to extend Rapaport's result to all ergodic measures in a fashion similar to \cite{JordanRapaport}, as the latter relies crucially on Shmerkin's result on $L^q$ dimension of self-similar measures with ESC \cite{Shmerkinlq}, which does not hold for IFS consisting of analytic maps, see the discussion in \cite[Section~1.2.2]{BaranyKolossvaryTroscheit}. Our result \eqref{eq: simult IFS} is able to deal with all ergodic measures and requires weaker regularity ($C^{1+\theta}$ instead of analyticity) than Rapaport's \cite{RapaportAnalytic}, at the cost of assuming transversality instead of ESC.

\subsection{Applications III - orthogonal projections of ergodic measures on self-conformal sets}

Applications of the universal projection theorem (Theorem \ref{thm: general proj}) require verifying the relative dimension separability property of a given family $\Mk$ of measures on a set $X$. In particular, we can verify it in two important cases of dynamical origin, related to the IFS theory discussed above:
\begin{enumerate}[(i)]
	\item\label{it: ergodic} for $\Mk$ being the family of all ergodic invariant measures on a self-conformal set $X$ satisfying the strong separation condition,
	\item\label{it: gibbs} for $\Mk$ being the family of all Gibbs measures (corresponding to all H\"older continuous potentials) on a self-conformal set $X$ (without any separation conditions).
\end{enumerate}

This allows us to conclude that Marstand's projection theorem \eqref{eq: marstrand} holds simultaneously for $\mu \in \Mk$, with $\Mk$ as in \ref{it: ergodic} and \ref{it: gibbs}. See Theorems \ref{thm: orth proj ergodic} and \ref{thm: main IFS} for precise statements. This can be compared with a result of Bruce and Jin \cite[Theorem 1.3]{BruceJinGibbs}, showing that given a self-conformal IFS on $\R^n$ which satisfies certain irrationality condition for the derivatives, one has $\hdim P_V \mu = \min\{ d, \hdim \mu\}$ for every $V \in \Gr(d,n)$ and every Gibbs measure $\mu$ (extending the results of Hochman and Shmerkin \cite{HochmanShmerkinLocal}). Our Theorem \ref{thm: main IFS} does not require any irrationality condition, at the cost of allowing the dimension formula to fail for a zero-measure set of projections.

\section{Main result: a universal projection theorem}\label{sec: main}

We shall now describe the setting in which we will prove a universal projection theorem. Let $X$ be a compact topological space (the \textit{phase space}) and let $U$ be a hereditary Lindel\"of topological space\footnote{topological space $U$ is hereditary Lindel\"of if every subset has the property that its every open cover has a countable subcover.  Every separable metric space is hereditary Lindel\"of, see e.g. \cite{EngelkingGT}.} (the \textit{parameter space}). Let $\eta$ be a locally finite Borel measure on $U$. For $\lam \in U$ let $\rho_\lam$ be a metric on $X$ compatible with its topology and let $B_\lam(x,r)$ denote the open $r$-ball in metric $\rho_\lam$. Fix $d \in \N$ and for each $\lam \in U$ consider a map $\Pi_\lam : X \to \R^d$ (\textit{projection}). Let $\Mk_\fin(X)$ denote the set of all finite Borel measures on $X$ and consider a collection $\Mk \subset \Mk_\fin(X)$. Our goal is to study projections $\Pi_\lam \mu, \lam \in U, \mu \in \Mk$ and provide projection theorems which hold for $\eta$-a.e. $\lam \in U$ and \textit{all} measures $\mu \in \Mk$ simultaneously. A crucial property needed for the proofs is a certain separability-like condition for the set $\Mk$ and a transversality condition for the family $\Pi_\lam$. For the former one, we make the following definitions.

\begin{defn}
Let $\mu$ and $\nu$ be finite Borel measures on a metric space $(X,\rho)$. The \textbf{relative dimension} of $\mu$ with respect to $\nu$ is
\[ \dim(\mu || \nu, \rho) := \inf \left\{ \eps > 0 : - \eps < \underset{x \sim \mu}{\essinf} \liminf \limits_{r \to 0} \frac{\log \frac{\mu(B(x,r))}{\nu(B(x,r))}}{\log r} \leq  \underset{x \sim \mu}{\esssup} \limsup \limits_{r \to 0} \frac{\log \frac{\mu(B(x,r))}{\nu(B(x,r))}}{\log r} < \eps  \right\}, \]
where $B(x, r)$ denotes the open $r$-ball in metric $\rho$.
\end{defn}

\begin{rem}
We adopt the convention that $\log \frac{\mu(B(x,r))}{0} = + \infty$ if $\mu(B(x,r)) > 0$. In particular $\dim(\mu || \nu, \rho) = \infty$ if $\mu(X \setminus \supp(\nu)) > 0$.
\end{rem}

\begin{defn}\label{defn: relative dimension separability}
Let $(X,\rho)$ be a metric space. Let $\Mk \subset \Mk_\fin(X)$ be a collection of finite Borel measures on $X$. We say that $\Mk$ is \textbf{relative dimension separable} (with respect to $\rho$)  if there exists a countable set $\Vk \subset \Mk_\fin(X)$ such that for every $\mu \in \Mk$ and $\eps > 0$ there exists $\nu \in \Vk$ with $\dim(\mu || \nu, \rho) < \eps$.
\end{defn}

We will also require measures in $\Mk$ to satisfy a mild regularity condition.

\begin{defn}
A finite Borel measure $\mu$ on a metric space $(X,\rho)$ is \textbf{weakly diametrically regular} if
\[ \lim \limits_{r \to 0} \frac{\log \frac{\mu(B(x,r))}{\mu(B(x, 2r))}}{\log r} = 0 \text{ for } \mu\text{-a.e. } x \in X. \]
\end{defn}

\begin{rem}
It is easy to see that every measure for which the local dimension exists at $\mu$-a.e. point (see Definition \ref{defn:dim}) is weakly diametrically regular. Moreover, any finite Borel measure on $\R^n$ is weakly diametrically regular (with respect to the Euclidean metric) - see \cite[Lemma 1]{barreira2001hausdorff}.
\end{rem}
Our principal assumptions on the families $\{ \rho_\lam : \lam \in U \},\ \{ \Pi_\lam : \lam \in U \}$ and measure $\eta$ are as follows:

\begin{enumerate}[start=1,label={(A\arabic*)}]
\item\label{it: metric compability} for every $\lam_0 \in U$ and every $\xi \in (0,1)$ there exists a neighbourhood $U'$ of $\lam_0$ and a constant $0 < H = H(\xi, \lam_0) < \infty$ such that for every $\lam \in U'$
\[ H^{-1} \rho_\lam(x,y)^{1+\xi} \leq \rho_{\lam_0}(x,y) \leq H \rho_\lam(x,y)^{1 - \xi} \text{ holds for every } x,y \in X,\]
\item\label{it: Lip} for each $\lam \in U$, map $\Pi_\lam$ is Lipschitz in $\rho_\lam$,
\item\label{it: trans} for every $\lam_0 \in U$ and $\eps>0$ there exist a neighbourhood $U'$ of $\lam_0$ and a constant $K = K(\lam_0, \eps)$ such that for every $x,y \in X,\ r>0, \delta>0$
\[ \eta(\{ \lam \in U' : |\Pi_\lam (x) - \Pi_\lam (y)| < \rho_\lam (x,y) r \text{ and } \rho_\lam (x,y) \geq \delta \}) \leq K \delta^{-\eps} r^{d-\eps}.\]

\end{enumerate}

Assumption \ref{it: trans} is a generalization of the classical transversality condition (see e.g. \cite[Section 14.4]{BSS} or \cite{SolomyakTransversSurvey}). We will provide an embedding theorem with the following regularity property for the embedding map.

\begin{defn}\label{defn: nearly biLip}
Let $(X,\rho_X)$ and $(Y, d_{Y})$ be metric spaces.
Let $\Pi: X \to Y$ be a Lipschtiz map and let $\mu$ be a finite Borel measure on $X$. We say that $\Pi$ is \textbf{$\mu$-nearly bi-Lipschitz} if $\mu$-a.e. $x \in X$ has the property that for every $\alpha \in (0,1)$ there exists $C=C(x,\alpha)$ such that
\[ \rho_X(x,y) \leq C \rho_Y(\Pi(x), \Pi(y))^\alpha \text{ for every } y \in X. \]
\end{defn}

The main result of this paper is the following. For the definitions of Hausdorff and Assouad dimensions see Section \ref{sec: prelim}.

\begin{thm}\label{thm: general proj}
Let $X$ be a compact topological space and let $U$ be a hereditary Lindel\"of topological space. Let $\{ \rho_\lam : \lam \in U \}$ be a family of metrics on $X$ compatible with its topology and let $\{ \Pi_\lam : \lam \in U \}$ be a family of maps $\Pi_\lam : X \to \R^d$ satisfying assumptions \ref{it: metric compability} - \ref{it: trans}. Let $\Mk \subset \Mk_\fin(X)$ be a collection of finite Borel measures on $X$, such that for every $\rho_\lam,\ \lam \in U$, the family $\Mk$ is relative dimension separable and each $\mu \in \Mk$ is weakly diametrically regular. Then for $\eta$-a.e. $\lam \in U$, the following holds simultaneously for all $\mu \in \Mk$:
\begin{enumerate}
	\item\label{it: theorem general hdim} $\lhdim \Pi_\lam \mu = \min\{ d, \lhdim(\mu, \rho_\lam)\}$ and $\uhdim \Pi_\lam \mu = \min\{ d, \uhdim(\mu, \rho_\lam)\}$,
	\item\label{it: theorem general adim} if $\adim(X, \rho_\lam) < d$, then $\Pi_\lam$ is $\mu$-nearly bi-Lipschitz in metric $\rho_\lam$.
\end{enumerate}
\end{thm}

Here and in the rest of the paper, the precise meaning of \textit{simultaneously for all} $\mu \in \Mk$ is, for instance in the case of the first part of point \eqref{it: theorem general hdim},
\[ \eta \left( \left\{ \lam \in U : \underset{\mu \in \Mk}{\exists}\   \lhdim \Pi_\lam \mu \neq \min\{ d, \lhdim(\mu, \rho_\lam)\} \right\} \right) = 0\]
and likewise for the other statements

\section{Applications}

Let us begin with listing several families of maps $\Pi_\lam$ and metrics $\rho_\lam$ satisfying conditions \ref{it: metric compability} - \ref{it: trans}.
 
\subsection{Orthogonal projections}
 
 \begin{example}[Orthogonal projections]\label{ex: orthogonal proj}
 Fix $1 \leq d < n$. Let $\Gr(d,n)$ be the Grassmannian of $d$-dimensional linear subspaces in $\R^n$. Let $\eta$ be the unique Borel probability measure on $\Gr(d,n)$ which is invariant under the action of the orthogonal group $O(d)$, see \cite[Section 3.9]{mattila}. Given a compact set $X \subset \R^n$, for $V \in \Gr(d,n)$ let $P_V : X \to \R^d$ denote the orthogonal projection onto $V$ (which one can identify with $\R^d$). Let $\rho$ be the Euclidean metric on $X$.  Setting $U = \Gr(d,n)$ and $\rho_V = \rho$ we obtain families $\{ P_V : V \in U \},\ \{ \rho_V : V \in U \}$ satisfying \ref{it: metric compability} - \ref{it: trans}, with \ref{it: trans} following from \cite[Lemma 3.11]{mattila}. That is, there exists a constant $c>0$
 \begin{equation}\label{eq:orthotrans}
 	\eta\left(\left\{V\in\Gr(d,n):\|P_V(x)-P_V(y)\|\leq r\|x-y\|\right\}\right)\leq c\cdot\min\{1,r^d\}
 \end{equation}
for every $x,y\in X$ and $r>0$.
 \end{example}
 
 Given a compact set $X \subset \R^n$ one cannot expect the family of all finite Borel measures on $X$ to be relative dimension separable, as conclusions of Theorem \ref{thm: general proj} might fail for it.
 
 \begin{example}\label{eq: orthogonal proj non example}
 Let $X$ be the closed unit disc in $\R^2$ centred at zero. For $V \in \Gr(1,2)$, let $\mu_V$ be the $1$-dimensional Hausdorff measure restricted to the unit interval passing through the origin and perpendicular to the subspace $V$. Then clearly $\hdim \mu_V =1$ but $\hdim P_V \mu_V = 0$. Therefore, by Theorem \ref{thm: general proj}, the family $\{\mu_V : V \in \Gr(1,2)\}$ is not relative dimension separable and hence neither is the larger family $\Mk_\fin(X)$. 
 \end{example}
 
 \subsection{Conformal IFS}
 
 One may find relative dimension separable families within measures of dynamical origin. Our main focus is on ergodic measures on self-conformal sets. Let us now describe those.
 
 For a compact connected set $V \subset \R^n$ with $V = \overline{\Int(V)}$, a function $f: V \to V$ is called a conformal $C^{1+\theta}$ map if it extends to a diffeomorphism $f : W \to W$ of an open connected set $W \supset V$ such that for every $x \in V$ the differential $f'(x) = D_x f$ is a non-singular similitude and the map $V \ni x \mapsto D_x f$ is $\theta$-H\"older for some $\theta > 0$. Note that in this case the operator norm $\| D_x f \|$ is simply the corresponding coefficient of similarity.
 
 \begin{example}[Conformal IFS]\label{ex: conformal IFS}
Let $\Ak$ be a finite set and let $V \subset \R^n$ be a compact connected set with $V = \overline{\Int(V)}$. For each $i \in \Ak$, let $f_i : V \to V$ be a conformal $C^{1+\theta}$ map such that $0 < \|f'_i(x)\| < 1$ for every $x \in V$. We call the collection $\Fk = (f_i)_{i \in \Ak}$ a \textbf{conformal $C^{1+\theta}$ IFS}. Let $\Sigma = \Ak^\N$ be the symbolic space over the alphabet $\Ak$. We can associate to $\Fk$ a \textbf{natural projection map} $\Pi_\Fk : \Sigma \to V$ defined as
\begin{equation}\label{eq: natural projection} \Pi_\Fk(\om) = \lim \limits_{n \to \infty} f_{\om_1} \circ f_{\om_2} \circ \cdots \circ f_{\om_n}(x),
\end{equation}
where $\om = (\om_1, \om_2, \ldots)$ and $x$ is any point in $V$. The set $\Lambda = \Pi_\Fk(\Sigma)$ is called the \textbf{attractor} of $\Fk$ and it is the unique non-empty compact set $\Lambda$ satisfying
\[ \Lambda = \bigcup \limits_{i \in \Ak} f_i(\Lambda). \]
Any set $\Lambda$ of this form is called a \textbf{self-conformal set}. An \textbf{ergodic measure on $\Lambda$} is a measure of the form $\Pi_\Fk \mu$, where $\mu$ is an ergodic shift-invariant Borel probability measure on $\Ak^\N$. If $\mu$ is additionally a Gibbs measure corresponding to a H\"older continuous potential (see Definition \ref{defn: Gibbs measure}), then $\Pi_\Fk \mu$ is called a \textbf{Gibbs measure on $\Lambda$}, while if $\mu$ is a Bernoulli measure then $\Pi_\Fk \mu$ is called a \textbf{self-conformal measure}. We say that $\Fk$ satisfies the \textbf{Strong Separation Condition} if sets $f_{i}(\Lambda), i \in \Ak$ are pairwise disjoint. We denote by $\Ek_\sigma(\Sigma)$ and $G_\sigma(\Sigma)$ the collections of all, respectively, ergodic and Gibbs measures on $\Sigma$. If $\Fk$ consists of similarity maps, i.e. $f_i(x) = \lam_i O_i x + t_i$, where $\lam_i \in (0,1), t_i \in \R^n$ and $O_i$ are orthogonal $n \times n$ matrices, then the corresponding attractor is called a \textbf{self-similar set} and $\Pi_\Fk \mu$ is called a \textbf{self-similar measure} if $\mu$ is Bernoulli.

Given a conformal $C^{1+\theta}$ IFS $\Fk$, it will be convenient for us to consider an associated metric $\rho_{\Fk}$ on $\Sigma$ defined as follows. Given two infinite sequences $\om,\tau \in \Sigma$, let $\om \wedge \tau$ denote the longest common prefix of $\om$ and $\tau$. For a finite word $\om \in \Sigma^* = \bigcup \limits_{k=0}^\infty \Ak^k$, let $f_\om = f_{\om_1} \circ\cdots \circ f_{\om_k}$, where $\om = (\om_1, \ldots, \om_k)$. If now $\om,\tau \in \Sigma$ and $\om \neq \tau$, we set
\begin{equation}\label{eq: adapted metric} \rho_\Fk(\om, \tau) = \|f'_{\om \wedge \tau} \|,
\end{equation}
where $\| \cdot \|$ is the supremum norm on $V$. It is easy to check that $\rho_\Fk$ is a metric on $\Sigma$ and the natural projection map $\Pi_\Fk : \Sigma \to \R^n$ is Lipschitz in $\rho_\Fk$. If $\Fk$ satisfies the Strong Separation Condition, then $\Pi_\Fk$ is bi-Lipschitz in $\rho_\Fk$ (see \cite[Section 14.2]{BSS}). The metric $\rho_\Fk$ is also significant as natural dynamical invariants of the system can be expressed in terms of dimensions calculated with respect to $\rho_\Fk$. Namely, an ergodic shift-invariant measure $\mu$ on $\Sigma$ is exact-dimensional with respect to $\rho_\Fk$ and satisfies
\begin{equation}\label{eq: dim symbolic}  \hdim (\mu, \rho_\Fk) = \frac{h(\mu)}{\chi(\mu, \Fk)},
\end{equation}
where $h(\mu)$ is the Kolmogorov-Sinai entropy of $\mu$ with respect to the left shift $\sigma : \Sigma \to \Sigma$ and $\chi(\mu, \Fk)$ is the Lyapunov exponent of $\mu$ defined as
\[ \chi(\mu, \Fk) = - \int \log \|f'_{\om_1}(\Pi_\Fk(\sigma \om))\|d\mu(\om).\]
Similarly, let us define the pressure function $P_\Fk : (0, \infty) \to \R$ as
\begin{equation}\label{eq: pressure def}
	P_\Fk(s) = \lim_{n\to\infty}\frac{1}{n}\log\left(\sum_{\om\in\Ak^n}\|f'_\om\|^{s}\right).
\end{equation}
It is well-known that there exists a unique solution $s(\Fk)$ of the equation (the so-called Bowen's formula)
\begin{equation}\label{def: Bowen formula}
P_\Fk(s(\Fk)) = 0,
\end{equation}
which satisfies
\[ \hdim(\Sigma, \rho_\Fk) = \adim(\Sigma, \rho_\Fk)  = s(\Fk). \]
The Hausdorff dimension being equal to $s(\Fk)$, and in particular, $\Sigma$ being $s(\Fk)$-Ahlfors regular, follows by Bedford \cite{Bedford_selfconf} and the assertion on the Assouad dimension (denoted by $\dim_A$) follows by the $s(\Fk)$-Ahlfors regularity and \cite[Theorem~6.4.1]{Fraser_book}.
\end{example}

Our basic example of a relative dimension separable collection of measures is the following one.

\begin{prop}\label{prop: rel dim sep IFS}
Let $\Fk$ be a conformal $C^{1+\theta}$ IFS on $\R^n$ and let $\rho_\Fk$ be the corresponding metric on the symbolic space $\Sigma$. Then the collection of all ergodic shift-invariant Borel probability measures on $\Sigma$ is relative dimension separable with respect to $\rho_\Fk$.
\end{prop}
The above proposition follows from a more general Proposition \ref{prop: general ergodic rds}. An immediate consequence, as the regular dimension separability is invariant under bi-Lipschitz mappings (see Proposition \ref{prop: rel dim sep under biLip}), is the following application of Theorem \ref{thm: general proj} with the use of transversality of orthogonal projections in Example \ref{ex: orthogonal proj}.

\begin{thm}\label{thm: orth proj ergodic}
Let $\Fk$ be a $C^{1+\theta}$ conformal IFS on $\R^n$ with attractor $\Lambda$, satisfying the Strong Separation Condition. Then the family of all ergodic measures on $\Lambda$ is relative dimension separable and hence for every $ 1 \leq d < n$, for almost every $V \in \Gr(d,n)$
\[ \hdim P_V \mu = \min \{ d, \hdim \mu \} \text{ simultaneously for all ergodic measures } \mu \text{ on } \Lambda. \]
\end{thm}

A less trivial application is the result saying that without assuming any separation, the same holds for all Gibbs measures on a self-conformal set.
\begin{thm}\label{thm: main IFS}
Let $\Lambda$ be a self-conformal set on $\R^n$. The collection of all Gibbs measures on $\Lambda$ is relative dimension separable. Consequently, for every $1 \leq d < n$, for almost every $V \in \Gr(d,n)$
\[ \hdim P_V  \mu = \min \{ d, \hdim\mu \} \text{ simultaneously for all Gibbs measures } \mu \text{ on } \Lambda. \]
\end{thm}
The above theorem follows from Theorem \ref{thm: general proj}, Example \ref{ex: orthogonal proj} and Propositions \ref{prop: Lipschitz urds to rds}, \ref{prop: Gibbs urds}.
 
 Let us now turn to \emph{parametrized} families of conformal IFS. Even in the presence of overlaps, one can obtain projection results for typical parameters if the transversality condition holds.

\begin{example}[Parametrized conformal IFS]\label{ex:parametrized IFS}
Let $\Ak$ be a finite set and let $V \subset \R^d$ be a compact connected set with $V = \overline{\Int(V)}$. For each $\lam \in U$, where $U$ is a (hereditary Lindel\"of) topological space, let $\Fk^\lam = (f_i^\lam)_{i\in \Ak}$ be a conformal $C^{1+\theta}$ IFS on $V$. Assume that there exist $0 < \gamma_1 < \gamma_2 < 1$ such that $\gamma_1 \leq \|(f^\lam_i)'(x)\| \leq \gamma_2$ for every $x \in V, \lam \in U$. Moreover, assume that the map $U \ni \lam \mapsto \Fk_\lam$ is continuous, where the distance between $\Fk^{\lam_1}$ and $\Fk^{\lam_2}$ is defined as $ \max \limits_{i \in \Ak} \left(\|f^{\lam_1}_i - f^{\lam_2}_i\| + \|(f^{\lam_1}_i)' - (f^{\lam_2}_i)'\| + \sup \limits_{x \neq y \in V} \frac{\|(f^{\lam_1}_i)'(x) - (f^{\lam_2}_i)'(y)\|}{|x-y|^\theta}\right)$. We call a parametrized family $\Fk^\lam, \lam \in U$ a \textbf{continuous family of $C^{1+\theta}$ conformal IFS}.
Let $\Sigma = \Ak^\N$ be the symbolic space, and for each $\lam \in U$, let $\Pi_\lam := \Pi_{\Fk^\lam} : \Sigma \to \R^d$ be the corresponding natural projection map \eqref{eq: natural projection}. For each $\lam \in U$, let $\rho_\lam := \rho_{\Fk^\lam}$ be the metric on $\Sigma$ corresponding to $\Fk^\lam$, as defined in \eqref{eq: adapted metric}. Under these assumptions families $\{ \Pi_\lam : \lam \in U \}$ and $\{\rho_\lam : \lam \in U\}$ satisfy assumptions \ref{it: metric compability} and \ref{it: Lip}. Property \ref{it: metric compability} follows by the bounded distortion property and the distortion continuity, see for example Nakajima \cite[Section~3]{Nak} in the more general non-autonomous case, while \ref{it: Lip} follows directly from the definition of metric $\rho_\lam$. 

It is straightforward to check that the transversality condition \ref{it: trans} with respect to the family of metrics $\rho_\lam$ follows from a stronger (and classical) condition: there exists $K$ such that
\begin{equation}\label{eq: classical trans} \eta\left( \left\{ \lam \in U : \|\Pi_\lam(\om) - \Pi_{\lam}(\tau)\| \leq r \right\} \right) \leq K\min\{1,r^d\} \text{ for every } \om,\tau \in \Sigma \text{ with } \om_1 \neq \tau_1.
\end{equation}
Indeed, by the bounded distortion property, there exists a continuous function $K\colon U\to (0, \infty)$ such that for every $x,y\in V$ and $\omega\in\Sigma_*$
 	$$
 	\|f_{\omega}^\lam(x)-f_{\omega}^\lam(y)\|\geq K(\lam)^{-1}\|(f_\omega^\lam)'\|\|x-y\|,
 	$$
 	see Nakajima \cite[Lemma~4.1]{Nak}.
 	And so,
 \[
 \begin{split}
 	\eta\left( \left\{ \lam \in U : \|\Pi_\lam(\om) - \Pi_{\lam}(\tau)\| \leq \rho_\lambda(\omega,\tau)r \right\} \right)&\leq\eta\left( \left\{ \lam \in U : \|\Pi_\lam(\sigma^{|\omega\wedge\tau|}\om) - \Pi_{\lam}(\sigma^{|\omega\wedge\tau|}\tau)\| \leq  K(\lambda)r \right\} \right)\\
 	&\leq\min\left\{1,r^d\sup_{\lambda\in U}K(\lambda)^d\right\}.
 \end{split}
 \]
\end{example}

It is usually a non-trivial task to check whether a given parametrized IFS satisfies the transversality condition \ref{it: trans}. We refer to \cite{BSS} for an overview of the technique and to \cite{SolomyakTransversSurvey} for a recent survey. Let us give a one simple construction which leads to a transversal family of conformal IFS.

\begin{example}[Translation family]\label{ex:translate}
Let  $\Fk = (f_i)_{i\in \Ak}$ be a conformal $C^{1+\theta}$ IFS on a compact connected set $V \subset \R^d$ as described in Example \ref{ex: conformal IFS} and assume that $\max \limits_{i \neq j \in \Ak} \|f_i'\| + \|f_j'\|<1$. Let $U=\{(t_i)_{ i \in \Ak}\in(\R^{d})^{\Ak}: f_i(V)+t_i\subset V\}$, and let $\eta$ be the normalized Lebesgue measure on $U$. Then $\Fk^{t}=\{f_i^t = f_i + t_i\}_{i\in\Ak}, t \in U$ is a continuous family of $C^{1+\theta}$ conformal IFS. A family of natural projections  $\{ \Pi_t : t \in U \}$ and corresponding metrics $\{\rho_t : t \in U\}$ as in Example \ref{ex:parametrized IFS}  satisfy \ref{it: metric compability} - \ref{it: trans} with measure $\eta$, with respect to the parameter $t$. See for example \cite[Theorem~14.5.2]{BSS} for the proof in the case $d=1$, which extends in a straightforward manner to higher dimensions.

Note also that if $d=1$ and $\Fk_t = \{ x \mapsto \lam_i x + t_i\}_{ i \in \Ak}$ consists of similarities with fixed $\lam_i \in (-1,1)\setminus\{0\}$ and parametrized by $t = (t_i)_{i \in \Ak} \in \R^{\Ak}$, then the similarity dimension $s_0 = s(\Fk_t)$ is independent of $t$ and in fact is the unique solution of the equation $\sum \limits_{i \in \Ak}|\lam_i|^{s_0} = 1$. Therefore, if $s_0 < 1$, then the condition $\max \limits_{i \neq j \in \Ak} \|f_i'\| + \|f_j'\|<1$ is satisfied and hence the family $\Fk_t$ satisfies the transversality condition. Note that for $d>1$, the analogous condition $s_0 < d$ is not sufficient for guaranteeing $\max \limits_{i \neq j \in \Ak} \|f_i'\| + \|f_j'\|<1$.
\end{example}

Before formulating our main result on parametrized families of IFS satisfying the transversality condition, let us interpret the nearly bi-Lipschitz condition for natural projection maps as a separation condition for cylinders. Given an infinite word $\om = (\om_1, \om_2, \ldots) \in \Sigma$, let $\om|_n = (\om_1, \ldots, \om_n)$ denote its restriction to the first $n$ coordinates, while given a finite word $\om \in \Sigma^*$, let $|\om|$ denote its length and let $[\om] = \{ \tau \in \Sigma : \tau|_n  = \om \}$ be the corresponding cylinder.

\begin{defn}\label{defn: ede} Let $\Fk$ be a $C^{1+\theta}$ conformal IFS on $\R^d$. We say that $\om \in \Sigma$ has \textbf{exponential distance from the enemy (EDE)} if for every $\varepsilon>0$ there exists $C=C(\om,\varepsilon)>0$ such that for every $n\in\N$
\begin{equation}\label{eq:ede}
\mathrm{dist}\left(\Pi_\Fk(\om),\bigcup_{\substack{|\tau|=n\\ \tau\neq \om|_n}}\Pi_\Fk([\tau])\right)>C\diam(\Pi_\Fk([\om|_n]))^{1+\varepsilon},
\end{equation}
\end{defn}

\begin{prop}\label{prop: ede holder} 	Let $\Fk$ be a $C^{1+\theta}$ conformal IFS on $\R^d$. Then $\om \in \Sigma$ has exponential distance from the enemy if and only if for every $\alpha \in (0,1)$ there exists $C = C(\om, \alpha)$ such that
\be\label{eq:holder}\rho_{\Fk} (\om,\tau) \leq C |\Pi_\Fk(\om) - \Pi_\Fk(\tau)|^\alpha \text{ for every } \tau \in \Sigma.
\ee
Consequently, for $\mu \in \Mk_{\fin}(\Sigma)$, we have that $\mu$-a.e. $\om \in \Sigma$ has exponential distance from the enemy if and only if the natural projection map $\Pi_\Fk$ is $\mu$-nearly bi-Lipschitz in metric $\rho_\Fk$.
\end{prop}

Therefore, for a measure $\mu \in \Mk_\fin(\Sigma)$, $\mu$-a.e. $\om \in \Sigma$ satisfies EDE if and only if $\Pi_\Fk$ is $\mu$-nearly bi-Lipschitz in metric $\rho_\Fk$. For the proof of Proposition \ref{prop: ede holder} see Section \ref{subsec: EDE}.

\begin{rem}\label{rem: ede useful}
Let us emphasize that EDE is a much weaker separation condition that the Open Set Condition (see \cite[Definition 1.5.2]{BSS}), so that it can hold for typical IFS with overlaps, while it is strong enough so that it can be used to deduce novel results on the multifractal analysis, see Section \ref{sec: multifractal}. In particular, see Lemma \ref{lem: nearly biLip loc dim} for a useful consequence of EDE, which is the main ingredient of the proof of Theorem \ref{thm:multi intro}.
\end{rem}

The following is our main result on transversal families of conformal IFS.

\begin{thm}\label{thm: main trans IFS}
Let $\Fk^\lam, \lam \in U$ be a continuous family of $C^{1+\theta}$ conformal IFS on $V \subset \R^n$. Assume that $\eta$ is a measure on $U$ such that the transversality condition \eqref{eq: classical trans} holds. Then for $\eta$-a.e. $\lam \in U$
\begin{enumerate}[(1)]
\item\label{it: main trans IFS dim} $\hdim \Pi_\lam \mu =  \min\left\{n, \frac{h(\mu)}{\chi(\mu, \Fk^\lam)} \right\}$ holds simultaneously for all ergodic measures $\mu$ on $\Sigma$,
\item\label{it: trasnsversal IFS orth proj} for every $1 \leq d < n$, for almost every $V \in \Gr(d, n)$, the equality $\hdim ( P_V \Pi_\lam \mu ) = \min\{d, \frac{h(\mu)}{\chi(\mu, \Fk^\lam)} \}$ holds simultaneously for all ergodic measures $\mu$ on $\Sigma$,
\item\label{it: main trans IFS EDE} if $s(\Fk^\lam) < n$, then simultaneously for all ergodic measures $\mu$ on $\Sigma$, $\mu$-a.e. $\om \in \Sigma$ has exponential distance from the enemy.
\end{enumerate}
\end{thm}

Note that point \ref{it: trasnsversal IFS orth proj} of Theorem \ref{thm: main trans IFS} asserts that under the transversality condition, despite possible overlaps, one obtains for typical $\Fk_\lam$ the same conclusion as in Theorem \ref{thm: main IFS} under Strong Separation Condition. Point \ref{it: main trans IFS dim} follows directly from Theorem \ref{thm: general proj} and Proposition \ref{prop: rel dim sep IFS}, while point \ref{it: main trans IFS EDE} requires additionally Proposition \ref{prop: ede holder}. 

Point \ref{it: trasnsversal IFS orth proj} requires an additional step to show that the map $P_V\circ\Pi_\lambda\colon\Sigma\mapsto\R^d$ satisfies the transversality condition with respect to the parametrisation $(V,\lambda)\in\Gr(d,n)\times U$. Then the claim of point~\ref{it: trasnsversal IFS orth proj} follows simply by Fubini's Theorem. Let us denote by $\gamma$ the unique measure defined in Example~\ref{ex: orthogonal proj}  and note that \eqref{eq: classical trans} gives that $\eta\left( \left\{ \lam \in U : \Pi_\lam(\om) = \Pi_\lam(\tau) \right\} \right) = 0$ whenever $\om \neq \tau$. Therefore for every $r>0$ by \eqref{eq:orthotrans} and \eqref{eq: classical trans}
\[
\begin{split}
&\gamma\times\eta\left(\left\{(V,\lambda):\ \|P_V\Pi_\lambda(\omega)-P_V\Pi_\lambda(\tau)\|< \rho_\lambda(\omega,\tau)r\right\}\right)\\
&\quad=\int\gamma\left(\left\{V:\|P_V\Pi_\lambda(\omega)-P_V\Pi_\lambda(\tau)\| < \|\Pi_\lambda(\omega)-\Pi_\lambda(\tau)\|\frac{\rho_\lambda(\omega,\tau)r}{\|\Pi_\lambda(\omega)-\Pi_\lambda(\tau)\|}\right\}\right)d\eta(\lambda)\\
&\quad\leq c\int\left(\frac{\rho_\lambda(\omega,\tau)r}{\|\Pi_\lambda(\omega)-\Pi_\lambda(\tau)\|}\right)^dd\eta(\lambda)= cr^d\int_0^\infty\eta\left(\left\{\lambda:\left(\frac{\rho_\lambda(\omega,\tau)r}{\|\Pi_\lambda(\omega)-\Pi_\lambda(\tau)\|}\right)^d>a\right\}\right)da\\
&\quad\leq c'r^d\int_0^\infty\min\left\{1,a^{-n/d}\right\}da\leq c''r^d.
\end{split}
\]

\subsection{Further examples}\label{sec: further examples}

Let us now present some more examples related to iterated function systems to which Theorem \ref{thm: general proj} can be applied.

\begin{example}[Parametrized non-autonomous system]
	Non-autonomous systems were introduced Rempe-Gillen and Urba\'nski
	in \cite{RU}, with the so-called Open Set Condition being assumed. For the
	sake of studying the overlapping non-autonumous systems, Nakajima \cite{Nak} considered families of non-autonomous systems and introduced a transversality condition for the non-autonomous systems. In particular the following parametrized family of non-autonomous system (with overlaps)
	was studied by Nakajima \cite{Nak}:
	The parameter domain is $U=\{t\in\mathbb{C}:|t|<2\times 5^{-5/8}\}\setminus\R$,
	and let
	$X:=\left\{z\in\mathbb{Z}:|z|\leq \frac{1}{1-2\cdot 5^{-5/8}}
	\right\}$.
	For every parameter $t\in U$ and $j\in\mathbb{N}$ Nakajima considered the mappings
	\begin{equation}
		\label{x98}
		\phi _{0,t}^{(j) },\phi _{1,t}^{(j) }:X\to X,\quad
		\quad
		\phi _{0,t}^{(j) } (z):=tz\quad \text{ and }\quad
		\phi _{1,t}^{(j) } (z):=tz+\frac{1}{j}.
	\end{equation}
	Let $\Phi  _{t}:=\left(\Phi  _{t}^{(j) }\right)_{j=1}^{\infty   }$, where
	\begin{equation}
		\label{x97}
		\Phi  _{t}^{(j) } =
		\left\{z\mapsto tz,z\mapsto tz+\frac{1}{j}\right\}
	\end{equation}
	Then $\Phi  _{t}$ is a parametrized non-autonomous system whose attractor is
	the set of all limit points
	$$
	\Lambda_t = \left\{
	\lim\limits_{n\to\infty} \phi _{\om_1,t}^{(1) }\circ\cdots\circ\phi _{\om_n,t}^{(n) } (z):\ (\om_1,\om_2,\dots  )\in \left\{0,1\right\}^{\mathbb{N}}
	\right\},$$
	where $z\in X$.	Set
	$\Ak=\{0,1\},\ \Sigma = \Ak^\N$ and $\Pi_t : \Sigma \to \mathbb{C},\ \Pi_{t}(\om)=\sum_{k=1}^\infty\frac{\om_kt^k}{k}$ for
	$\om=(\om_1,\om_2,\dots  )\in\Sigma $, so that $\Lambda_t = \Pi_t(\Sigma)$. Consider a family of metric on $\Sigma$ defined as	$\rho_t(\om, \tau)=t^{|\om \wedge \tau|}$ for $\om, \tau \in \Sigma$.
	It is straightforward to that conditions \ref{it: metric compability} and \ref{it: Lip} hold. It  follows from \cite[Theorem~B]{Nak} that the transversality condition \ref{it: trans} holds on every compact subset of $U$ with $\eta$ being the Lebesgue measure on $\mathbb{C}$. Indeed, by \cite[Theorem~B]{Nak}, if $\om, \tau \in \Sigma$ are such that $|\om \wedge \tau| = n$, then for a compact set $G \subset U$
	\[ \eta\left( \left\{ t \in G : |\Pi_t(\sigma^n \om) - \Pi_t(\sigma^n \tau)| \leq r  \right\} \right) \leq C_n r^2 \]
	with $C_n$ satisfying $\lim \limits_{n \to \infty} \frac{\log C_n}{n} = 0$. Therefore for every $\eps > 0$ there exists $K$ so that $C_n \leq K 2^{n \eps}$ and hence setting $\gamma = 2\times 5^{-5/8}$ we have for $r > 0, \delta > 0$
	\[
	\begin{split} & \eta\left( \left\{ t \in G : |\Pi_t(\om) - \Pi_t(\tau)| \leq \rho_t(\om, \tau)r,\ \rho_t(\om, \tau) \geq \delta  \right\} \right) \\
	& \qquad = \eta\left( \left\{ t \in G : |\Pi_t(\sigma^n\om) - \Pi_t(\sigma^n\tau)| \leq r,\ |t|^{n} \geq \delta  \right\} \right) \\
	& \qquad \leq \eta\left( \left\{ t \in G : |\Pi_t(\sigma^n\om) - \Pi_t(\sigma^n\tau)| \leq r,\ \gamma^{n} \geq \delta  \right\} \right) \\
	& \qquad \leq \eta\left( \left\{ t \in G : |\Pi_t(\sigma^n\om) - \Pi_t(\sigma^n\tau)| \leq r \right\} \right) \mathds{1}_{\left\{n \leq \frac{\log \delta}{\log \gamma} \right\}} \\
	& \qquad \leq C_n r^2 \mathds{1}_{\left\{n \leq \frac{\log \delta}{\log \gamma} \right\}} \leq K 2^{n\eps} r^2 \mathds{1}_{\left\{n \leq \frac{\log \delta}{\log \gamma} \right\}} \leq K \delta^{\frac{\eps}{\log \gamma}} r^2.
	\end{split}
	\]
	Consequently  \ref{it: trans} holds in this case. Finally, the set of all ergodic shift-invariant measures on $\Sigma$ is relative dimensional separable with respect to each $\rho_t$ with $t \in U$ by Proposition \ref{prop: general ergodic rds}, hence all assumptions of Theorem \ref{thm: general proj} are met in this case.
\end{example}

\begin{example}[Random self-similar system]
		Let $\Ak$ be a finite set of indices, and for every $i\in\Ak$, let $\theta_i\in(0,1)$, $O_i\in O(\R,d)$ and $t_i\in\R^d$. Let $I\subset\R^d$ be a compact domain and let $\zeta$ be the normalized Lebesgue measure on $I$. Let $U=I^{\Sigma_*}$. Set $\Sigma = \Ak^\N$ and for $\lam \in U$ define $\Pi_\lam : \Sigma \to \R^d$ as $\Pi_\lam(\om)=\sum_{k=1}^{\infty}(t_{\om_k}+\lam_{\om|_{k}})\theta_{\om|_{k-1}}O_{\om_1\dots  \om_k}$, where $\theta_{\om|_k}=\theta_{\om_1}\cdots\theta_{\om_k}$ and $O_{\om_1\dots \om_k}:= O_{\om_1}\cdots O_{\om_k}$. With the choice of a single metric $\rho_\lam(\om,\tau) = \rho(\om, \tau)=\theta_{\om \wedge \tau}$ on $\Sigma$, it is easy to check that the assumptions \ref{it: metric compability} and \ref{it: Lip} hold. Choosing the probability measure $\eta=\zeta^{\Sigma_*}$, the transversality condition \ref{it: trans} holds by \cite[Lemma~5.1]{JPS07} The set of all ergodic shift-invariant measures on $\Sigma$ is relative dimensional separable with respect to $\rho$ by Proposition \ref{prop: general ergodic rds}.
\end{example}

\subsection{Applications to the multifractal analysis}\label{subsec: multifractal}

Finally, let us discuss in more details the application of Theorem~\ref{thm: main trans IFS} to the multifractal analysis of self-similar measures. Let $\Fk_t=\{f_i(x)=\lambda_i O_i x+t_i\}_{i\in\Ak}$ be an IFS consisting of similarities on $\R^d$, so that $\lam_i \in (0,1)$, $O_i$ is a $d \times d$ orthogonal matrix and $t_i \in \R^d$. Let $p = (p_i)_{i \in \Ak}$ be a probability vector. For every $q \in \R$ we define $T(q)$ to be the unique real number satisfying
\[ \sum_{i\in\Ak}p_i^q \lambda_i^{T(q)}=1. \] 

The general version of our main result on the multifractal spectra of self-similar measures is the following.

\begin{thm}\label{thm: multi high dim main}
Fix $d \geq 1$. Let $\Ak$ be a finite set and for each $ i \in \Ak$ fix $\lam_i \in (0,1)$ and a $d \times d$ orthogonal matrix $O_i$. For each $t = (t_i)_{i \in \Ak} \in (\R^d)^{\Ak}$ and a probability vector $p = (p_i)_{i \in \Ak}$, let $\nu_{t,p}$ be the self-similar measure corresponding to the IFS $\Fk_t=\{f_i(x)=\lambda_i O_i x+t_i\}_{i\in\Ak}$ and the probability vector $p$. The following holds for Lebesgue almost every $(t_i)_{i\in\Ak} \in (\R^d)^\Ak$. If the similarity dimension $s_0 = s(\Fk_t)$ of $\Fk_t$ satisfies
\begin{enumerate}[(i)]
	\item $s_0 < 1$ if $d=1$, or
	\item $s_0 < d$ and $\max \limits_{i \in \Ak} \lambda_i<1/2$ if $d \geq 2$,
\end{enumerate}
then equality
	\[\dim_H \left( \left\{x:d(\nu_{t,p}, x)=\alpha\right\} \right)=\inf_{q\in\R}(\alpha q+T(q))\]
holds for
\begin{enumerate}[(i)]
	\item every $\alpha\in\left[\min\limits_{i\in\Ak}\frac{\log p_i}{\log\lambda_i},\max\limits_{i\in\Ak}\frac{\log p_i}{\log\lambda_i}\right]$ if $d = 1$, or
	\item every $\alpha\in\left[\frac{\sum \limits_{i \in \Ak} \lambda_i^{s_0}\log p_i}{\sum \limits_{i \in \Ak} \lambda_i^{s_0}\log\lambda_i},\max\limits_{i\in\Ak}\frac{\log p_i}{\log\lambda_i}\right]$ if $d \geq 2$.
\end{enumerate}
and every probability vector $p$ such that $p_i \neq \lambda_i^{s_0}$ for some $i \in \Ak$.
\end{thm}

We shall make now several comments on the above result. 
\begin{rem}\label{rem: L^q}
Theorem \ref{thm: multi high dim main} with $d > 1$ does not cover full range $\left[\min\limits_{i\in\Ak}\frac{\log p_i}{\log\lambda_i},\max\limits_{i\in\Ak}\frac{\log p_i}{\log\lambda_i}\right]$. This is because our technique deals only with decreasing part of the spectrum $\left[\frac{\sum \limits_{i \in \Ak} \lambda_i^{s_0}\log p_i}{\sum \limits_{i \in \Ak} \lambda_i^{s_0}\log\lambda_i},\max\limits_{i\in\Ak}\frac{\log p_i}{\log\lambda_i}\right]$, while the increasing part $\left[\min\limits_{i\in\Ak}\frac{\log p_i}{\log\lambda_i},\frac{\sum \limits_{i \in \Ak} \lambda_i^{s_0}\log p_i}{\sum \limits_{i \in \Ak} \lambda_i^{s_0}\log\lambda_i}\right]$ follows from the result of Barral and Feng \cite[Theorem~1.2, Remark~7.3]{BarralFeng}. As the latter was so far obtained only for $d = 1$, the full range result is limited to this case. Nevertheless, it follows from the proof of Theorem \ref{thm: multi high dim main} that in the remaining part of the spectrum we obtain a lower bound on the dimension of the level set, i.e.
\[\dim_H \left( \left\{x:d(\nu_{t,p}, x)=\alpha\right\} \right) \geq \inf_{q\in\R}(\alpha q+T(q))\]
holds for $\alpha \in \left[\min\limits_{i\in\Ak}\frac{\log p_i}{\log\lambda_i},\frac{\sum \limits_{i \in \Ak} \lambda_i^{s_0}\log p_i}{\sum \limits_{i \in \Ak} \lambda_i^{s_0}\log\lambda_i}\right]$.
\end{rem}

\begin{rem}
We shall emphasize that Theorem \ref{thm: multi high dim main}, even for $d=1$, does not establish the full \textit{multifractal formalism} for the corresponding self-similar measures. One says that the multifractal formalism holds at $\alpha$ if \[\dim_H \left( \left\{x:d(\nu_{t,p}, x)=\alpha\right\} \right)=\inf_{q\in\R}(\alpha q-\tau_{\nu_{t,p}}(q)),\]
where $q \mapsto \tau_{\nu_{t,p}}(q)$ is the $L^q$-spectrum of the self-similar measure $\nu_{t,p}$, see e.g. \cite[Chapter 5]{BSS}. Our result shows that the dimension of the level sets of the local dimension is preserved under the projection from the symbolic space onto the self-similar set. Establishing the multifractal formalism in the setting of Theorem \ref{thm: multi high dim main}, requires proving additionally that $\tau_{\nu_{t,p}}(q) = -T(q)$ for $q \in \R$. For $q>0$ this was proved by Shmerkin for self-similar measures satisfying the exponential separation condition \cite{Shmerkinlq}. The case $q < 0$ remains an open problem. The results of Barral and Feng \cite{BarralFeng} establish the full multifractal formalism for $d=1$ in the restricted range $\alpha \in \left[\min\limits_{i\in\Ak}\frac{\log p_i}{\log\lambda_i},\frac{\sum \limits_{i \in \Ak} {\b \lam_i^{s_0}}\log p_i}{\sum \limits_{i \in \Ak} {\b \lam_i^{s_0}}\log\lambda_i}\right]$ (corresponding to $q>0$), for self-similar measures on $\R$ satisfying the exponential separation condition.
\end{rem}

\section{Preliminaries}\label{sec: prelim}

For a map $\Pi$ between metric spaces, we will denote by $\Lip(\Pi)$ the Lipschitz constant of $\Pi$ (or $\Lip(\Pi, \rho)$ if we want to emphasize dependence on the metric). Given a Borel measure $\mu$ on a metric space $(X,\rho)$, we denote by $\supp(\mu)$ the topological support of $\mu$, i.e. the set of all $x \in X$ such that $\mu(B(x,r)) > 0$ for every $r>0$.

\begin{defn}\label{defn:dim}
	Let $(X, \rho)$ be a metric space. The \textbf{lower} and \textbf{upper local dimensions} of a finite Borel measure $\mu \in \Mk_\fin(X)$ at a point $x \in \supp ( \mu )$ are defined as
	\[ \ld(\mu, x) = \liminf \limits_{r \to 0} \frac{\log \mu(B(x,r))}{\log r} \text{ and } \ud(\mu, x) = \limsup \limits_{r \to 0} \frac{\log \mu(B(x,r))}{\log r}. \]
	If the limit above exists, then their common value is called the \textbf{local dimension} of $\mu$ at $x$ and denoted as $d(\mu, x)$. The \textbf{lower} and \textbf{upper Hausdorff dimensions} of $\mu$ are defined as
	\[ \lhdim \mu = \underset{x \sim \mu}{\essinf}\ \ld(\mu, x) \text{ and } \uhdim \mu = \underset{x \sim \mu}{\esssup}\ \ld(\mu,x). \]
	If $\lhdim \mu = \uhdim \mu$, then their common value is called the \textbf{Hausdorff dimension} of $\mu$ and denoted $\hdim \mu$. Measure $\mu$ is called \textbf{exact-dimensional} if $d(\mu,x)$ exists  and is constant $\mu$-almost everywhere.
\end{defn}

It is well known (see e.g. \cite[Section 1.9.1]{BSS}) that the Hausdorff dimensions can be equivalently expressed as
\[ \lhdim\mu = \inf \left\{ \hdim A  : A \subset X \text{ Borel with } \mu(A) > 0 \right\}\]
and
\[ \uhdim\mu = \inf \left\{ \hdim A : A \subset X \text{ Borel with } \mu(\R^N \setminus A) = 0 \right\}. \]

\begin{defn}\label{defn:adim}
Let $(X,\rho)$ be a metric space. For a set $Y \subset \R^N$ and $\delta>0$, let $N(Y, \delta)$ denote the minimal number of balls of radius $\delta$ required to cover $X$. Set $Y$ is said to be $(M,s)$-\emph{homogeneous} if $N(Y \cap B(x, r), \rho) \leq M(r / \rho)^s$ for every $x \in Y$, $0 < \rho < r$, i.e.~the intersection $B(x,r) \cap Y$ can be covered by at most $M(r / \rho)^s$ balls of radius $\rho$. The \textbf{Assouad dimension} of $Y$ is defined as
	\[ \adim Y = \inf \{ s > 0 : Y \text{ is } (M, s)\text{-homogeneous for some } M > 0 \}. \]
\end{defn}

We will repeatedly make use of the following simple lemma.

\begin{lem}\label{lem: packing cover}
Let $X$ be a metric space. Fix $r>0$. For every set $G \subset X$ there exists a cover
\[ G \subset \bigcup \limits_{x' \in F} B(x', r) \]
such that
\[ F \subset G \text{ and } \left\{ B(x', r/2) : x' \in F \right\} \text{ consists of pairwise disjoint sets}. \]
Moreover, if $X$ is separable, then $F$ can be taken to be countable and if $X$ is compact, then $F$ can be taken finite.
\end{lem}

\begin{proof}
As $F$ one can take a maximal packing $r/2$-packing of $G$, i.e. a set $F \subset G$ with the property that balls $\{ B(x', r/2) \}_{x' \in F}$ are pairwise disjoint and no larger set (in the sense of inclusion) has this property (its existence follows from the Kuratowski-Zorn lemma). Then $G \subset \bigcup \limits_{x' \in F} B(x', r)$, as otherwise $F$ would not be a maximal $r/2$-packing. If $X$ is separable, then it is hereditary Lindel\"of \cite{EngelkingGT}, hence one choose a countable subcover of $G$. If $X$ is compact, then $F$ must be finite, as otherwise $X$ would have a sequence without a convergent subsequence.
\end{proof}

\section{Relative dimension separability}\label{sec: rel dim sep}

We begin with simple observations providing conditions which guarantee that the relative dimension separability is preserved under a Lipschitz map.

\begin{prop}\label{prop: rel dim sep under biLip}
Let $X$ and $Y$ be metric spaces and let $\Pi: X \to Y$ be a bi-Lipschitz map. Assume that $\Mk \subset \Mk_\fin(X)$ is relative dimension separable and that for each $\mu \in \Mk$, measure $\Pi \mu$ is weakly diametrically regular. Then the collection $\{ \Pi \mu : \mu \in \Mk \} \subset \Mk_{\fin}(Y)$ is relative dimension separable.
\end{prop}

\begin{proof}
Let $\rho_X$ and $\rho_Y$ denote the metrics on $X$ and $Y$, respectively. Let $\Vk \subset \Mk_\fin(X)$ be a countable set witnessing relative dimension separability of $\Mk$. Fix $\eps>0,\ \mu \in \Mk$ and $\nu \in \Vk$ such that $\dim(\mu || \nu, \rho_X) < \eps$. As $\Pi$ is bi-Lipschitz, there exists $L>0$ such that for every $x \in X$ and $r>0$
\[ B(x, r/L) \subset \Pi^{-1}(B(\Pi(x),r)) \subset B(x, Lr). \]
Fix $x \in X$ such that there exist $R > 0$ and $d \geq 0$ for which
\[
	r^\eps \nu(B(x,r)) \leq \mu(B(x,r)) \leq r^{-\eps}\nu(B(x,r))
\]
and
\[
	\Pi\mu(B(\Pi(x), L^2 r)) \leq r^{-\eps} \Pi\mu(B(\Pi(x), r))
\]
hold for every $0 < r <R$. By assumptions, $\mu$-a.e. $x \in X$ satisfies the above properties. For $0 < r < R$ we have
\[\begin{split} \Pi\mu(B(\Pi(x), r)) & \leq  L^{2\eps} r^{\eps} \Pi\mu(B(\Pi(x), r/L^2)) \leq L^{2\eps} r^{-\eps} \mu(B(x,r/L)) \leq  L^{3\eps} r^{-2\eps} \nu(B(x,r/L)) \\
	& \leq L^{3\eps} r^{-2\eps} \nu(\Pi^{-1}(B(\Pi(x),r))) = L^{3\eps} r^{-2\eps} \Pi\nu(B(\Pi(x),r)).
\end{split}\]
Similarly, one has for $0<r < r/L^2$.
\[\begin{split} \Pi\mu(B(\Pi(x), r)) & \geq  r^{\eps}	\Pi\mu(B(\Pi(x), L^2 r)) \geq r^\eps \mu(B(x, Lr)) \geq L^\eps r^{2\eps} \nu(B(x, Lr)) \\
	& \geq L^\eps r^{2\eps} \nu(\Pi^{-1}(B(\Pi(x),r))) = L^\eps r^{2\eps} \Pi\nu(B(\Pi(x), r)).
\end{split}\]
The two above calculations show together that
\[ -2\eps \leq \liminf \limits_{r \to 0} \frac{\log \frac{\Pi\mu(B(\Pi(x),r))}{\Pi\nu(B(\Pi(x),r))}}{\log r}  \leq \limsup \limits_{r \to 0} \frac{\log \frac{\Pi\mu(B(\Pi(x),r))}{\Pi\nu(B(\Pi(x),r))}}{\log r} \leq 2\eps.\]
Recalling that the above holds for $\mu$-a.e. $x \in X$ we obtain
$\dim(\Pi\mu || \Pi\nu, \rho_Y) \leq 2\eps$. As the set $\{ \Pi\nu : \nu \in \Vk \}$ is at most countable, we see that $\{ \Pi\mu : \mu \in \Mk \}$ is relative dimension separable.
\end{proof}

For the next preservation property we need a stronger separability condition.

\begin{defn}
Let $\mu$ and $\nu$ be finite Borel measures on a metric space $(X,\rho)$. The \textbf{uniform relative dimension} of $\mu$ with respect to $\nu$ is
\[ \dim_u(\mu || \nu, d) := \inf \{ \eps > 0 : \underset{R > 0}{\exists}\ \underset{x \in X}{\forall}\ \underset{0 < r <  R}{\forall}\ r^\eps \nu(B(x,r)) \leq \mu(B(x,r)) \leq r^{-\eps}\nu(B(x,r))  \}. \]
\end{defn}

\begin{defn}\label{defn: urds}
Let $(X,\rho)$ be a metric space. Let $\Mk \subset \Mk_\fin(X)$ be a collection of finite Borel measures on $X$. We say that $\Mk$ is \textbf{uniform relative dimension separable} (with respect to $\rho$)  if there exists a countable set $\Vk \subset \Mk_\fin(X)$ such that for every $\mu \in \Mk$ and $\eps>0$ there exists $\nu \in \Vk$ such that $\dim_u(\mu || \nu, d) < \eps$.
\end{defn}

Note that clearly $\dim(\mu || \nu, d) \leq \dim_u(\mu || \nu, d)$, hence uniform relative dimension separability implies relative dimension separability. We shall also need a stronger diametric regularity condition.

\begin{defn}\label{defn: uniform weak diametric regularity}
Let $(X,\rho)$ be a metric space. Measure $\mu \in \Mk_\fin(X)$ is called \textbf{uniformly diametrically regular} if for every $\eps>0$ there exists $R>0$ such that
\[ \mu(B(x,2r)) \leq r^{-\eps}\mu(B(x,r)) \]
holds for every $0 < r < R$ and $x \in X$.
\end{defn}

\begin{prop}\label{prop: Lipschitz urds to rds}
Let $X$ and $Y$ be metric spaces and let $\Pi: X \to Y$ be a Lipschitz map. Assume that $X$ is separable, $\Mk \subset \Mk_\fin(X)$ is uniform relative dimension separable and that for every $\mu \in \Mk$, $\mu$ is uniformly diametrically regular and $\Pi\mu$ is weakly diametrically regular. Then the collection $\{ \Pi \mu : \mu \in \Mk \} \subset \Mk_{\fin}(Y)$ is relative dimension separable.
\end{prop}

\begin{proof} Let $\Vk \subset \Mk_\fin(X)$ be a countable set witnessing uniform relative dimension separability of $\Mk$. Fix $\eps>0,\ \mu \in \Mk$ and $\nu \in \Vk$ such that $\dim_u(\mu || \nu, \rho_X) < \eps$. Let $R_1>0$ be such that
\begin{equation}\label{eq: rel dim}
r^\eps \nu(B(x,r)) \leq \mu(B(x,r)) \leq r^{-\eps}\nu(B(x,r))
\end{equation}
and
\begin{equation}\label{eq: udr}
	\mu(B(x, 2r)) \leq r^{-\eps} \mu(B(x, r))
\end{equation}
hold for every $x \in X$ and $0 < r <R_1$. Fix $x\in X$ such that there exists $R_2 > 0$ so that
\begin{equation}\label{eq: wdr}
\Pi\mu(B(\Pi(x), 2r)) \leq r^{-\eps} \Pi\mu(B(\Pi(x), r))
\end{equation}
holds for every $0 < r < R_2$. By assumptions, $\mu$-a.e. $x \in X$ satisfies the above properties. Set $R= \min\{R_1, R_2 \}$ and $L =\Lip(\Pi)$. Using Lemma \ref{lem: packing cover}, take a countable cover
\[ \Pi^{-1}\left(B\left(\Pi\left(x\right), r/2\right)\right) \subset \bigcup \limits_{x' \in F} B\left(x', \frac{r}{L}\right)  \]
such that
\[ F \subset \Pi^{-1}\left(B\left(\Pi\left(x\right), r/2\right)\right)\text{ and } \left\{  B\left(x', \frac{r}{2L}\right) : x' \in F \right\} \text{ consists of pairwise disjoint sets}. \]
Note that
\[ \bigcup \limits_{x' \in F} B\left(x', \frac{r}{2L}\right) \subset \Pi^{-1}\left(B\left(\Pi\left(x\right), r\right)\right). \]
We therefore have by \eqref{eq: rel dim} and \eqref{eq: wdr} for $0 < r < \min\{ R, LR\}$
\[
\begin{split} \Pi\mu\left(B\left(\Pi\left(x\right),r\right)\right) & \leq 2^\eps r^{-\eps} \Pi\mu\left(B\left(\Pi\left(x\right), r/2\right)\right) \leq 2^\eps r^{-\eps} \sum \limits_{x' \in F} \mu\left(B\left(x', \frac{r}{L}\right)\right)  \\
	&  \leq 2^{2\eps} L^\eps  r^{-2\eps} \sum \limits_{x' \in F} \mu\left(B\left(x', \frac{r}{2L}\right)\right) \\
	& \leq 2^{3\eps} L^{2\eps}  r^{-3\eps} \sum \limits_{x' \in F} \nu\left(B\left(x', \frac{r}{2L}\right)\right) \\
	& \leq 2^{3\eps} L^{2\eps}  r^{-3\eps} \Pi\nu(B(\Pi(x), r)).
\end{split}\]

Similarly for $0 < r < \min\{ R, LR\}$
\[
\begin{split} \Pi\nu\left(B\left(\Pi\left(x\right),r/2\right)\right) &  \leq  \sum \limits_{x' \in F} \nu\left(B\left(x', \frac{r}{L}\right)\right) \leq  L^\eps   r^{-\eps}\sum \limits_{x' \in F} \mu\left(B\left(x', \frac{r}{L}\right)\right) \\
	&  \leq 2^{\eps} L^{2\eps} r^{-2\eps} \sum \limits_{x' \in F} \mu\left(B\left(x', \frac{r}{2L}\right)\right) \\
	& \leq 2^{\eps} L^{2\eps}  r^{-2\eps} \Pi\mu(B(\Pi(x), r))\\
	& \leq 2^{2\eps} L^{2\eps} r^{-3\eps} \Pi\mu(B(\Pi(x), r/2)).
\end{split}\]

The two above calculations show together that
\[ -3\eps \leq \liminf \limits_{r \to 0} \frac{\log \frac{\Pi\mu(B(\Pi(x),r))}{\Pi\nu(B(\Pi(x),r))}}{\log r}  \leq \limsup \limits_{r \to 0} \frac{\log \frac{\Pi\mu(B(\Pi(x),r))}{\Pi\nu(B(\Pi(x),r))}}{\log r} \leq 3\eps.\]
Recalling that the above holds for $\mu$-a.e. $x \in X$ we obtain
$\dim(\Pi\mu || \Pi\nu, \rho_Y) \leq 3\eps$. As the set $\{ \Pi\nu : \nu \in \Vk \}$ is at most countable, we see that $\{ \Pi\mu : \mu \in \Mk \}$ is relative dimension separable.
\end{proof}

\section{Proof of Theorem \ref{thm: general proj}}

This section is devoted to the proof of Theorem \ref{thm: general proj}. \textbf{We assume in this section that all assumptions of Theorem \ref{thm: general proj} are satisfied.} We will denote by $\Lip(\Pi_\lam, \rho_\lam)$ the Lipschitz constant of $\Pi_\lam : X \to \R^d$ with respect to metric $\rho_\lam$ on $X$.

\subsection{Preliminaries on relative dimension}

Let us begin with formulating the main technical consequence of the relative dimension separability assumption, which is a construction of sets on which one can compare $\mu$- and $\nu$-measures of balls for all $\mu$ which are relative dimension close to the reference measure $\nu$. For that, given $\mu, \nu \in \Mk_{\fin}(X),\ \lam_0 \in U, q \geq 0, \eps >0$ and $R > 0$ define sets
\[
\begin{split} A_{\lam_0, q, \eps, R}(\mu, \nu) = \bigg\{ x \in X : \underset{0 < r < R}{\forall}\ &r^\eps \nu(B_{\lam_0}(x,r)) \leq \mu(B_{\lam_0}(x,r)) \leq r^{-\eps} \nu(B_{\lam_0}(x,r)),\\
& \mu(B_{\lam_0}(x,r)) \leq r^{q - \eps},\ \mu(B_{\lam_0}(x,2r))  \leq r^{-\eps} \mu (B_{\lam_0}(x,r)) \bigg\}
\end{split}
\]
and
\[
G_{\lam_0, q, \eps, R}(\nu) = \bigg\{ x \in X : \underset{0 < r < R/2}{\forall}\ \nu(B_{\lam_0}(x,r)) \leq r^{q-2\eps},\ \nu(B_{\lam_0}(x,2r)) \leq 2^{-\eps}r^{-3\eps}\nu(B_{\lam_0}(x,r)) \bigg\}.
\]

A formal corollary of definitions of $ A_{\lam_0, q, \eps, R}(\mu, \nu)$ and $G_{\lam_0, q, \eps, R}(\nu)$ is the following.

\begin{lem}\label{lem: A G bounds}
For every $\mu, \nu \in \Mk_{\fin}(X),\ \lam_0 \in U, q \geq 0, \eps >0$ and $R>0$
\begin{equation}\label{eq: A subset G}
A_{\lam_0, q, \eps, R}(\mu, \nu) \subset G_{\lam_0, q, \eps, R}(\nu)
\end{equation}
and inequality
\begin{equation}\label{eq: mu nu comparison}
\mu (B_{\lam_0}(x,r) \cap A_{\lam_0, q, \eps, R}(\mu, \nu)) \leq 2^{-\eps}r^{-5\eps} \nu(B_{\lam_0}(x,r))
\end{equation}
holds for every $x \in G_{\lam_0, q, \eps, R}(\nu)$ and $0 < r < R/2$. Moreover, there exists $M=M(R, \eps)$ such that
\begin{equation}\label{eq: full set measure bound}
\mu( A_{\lam_0, q, \eps, R}(\mu, \nu)) \leq M \nu(X).
\end{equation}
\end{lem}
\begin{proof}
Containment \eqref{eq: A subset G} follows directly from the definitions of $A_{\lam_0, q, \eps, R}(\mu, \nu)$ and $G_{\lam_0, q, \eps, M}(\nu)$. For \eqref{eq: mu nu comparison}, fix $x \in G_{\lam_0, q, \eps, M}(\nu)$ and $0 < r < R/2$. If $B_{\lam_0}(x, r) \cap A_{\lam_0, q, \eps, R}(\mu, \nu) = \emptyset$, then \eqref{eq: mu nu comparison} holds trivially. Otherwise, choose $y \in B_{\lam_0}(x, r) \cap A_{\lam_0, q, \eps, R}(\mu, \nu)$. Then
\[\begin{split}
\mu (B_{\lam_0}(x,r) \cap A_{\lam_0, q, \eps, R}(\mu, \nu)) & \leq \mu(B_{\lam_0}(y, 2r)) \leq r^{-\eps}\mu(B_{\lam_0}(y, r)) \leq r^{-2\eps}\nu(B_{\lam_0}(y, r)) \\
& \leq r^{-2\eps}\nu(B_{\lam_0}(x, 2r)) \leq 2^{-\eps}r^{-5\eps}\nu(B_{\lam_0}(x, r)),
\end{split}\]
where  the second and third inequality follow from the definition of $A_{\lam_0, q, \eps, R}(\mu, \nu)$, the last inequality follows from the definition of $G_{\lam_0, q, \eps, M}(\nu)$ and the remaining ones follow from $B_{\lam_0}(x,r) \subset B_{\lam_0}(y, 2r)$ and $B_{\lam_0}(y,r) \subset B_{\lam_0}(x, 2r)$.

For \eqref{eq: full set measure bound} consider a countable cover
\[ A_{\lam_0, q, \eps, R}(\mu, \nu) \subset \bigcup \limits_{x' \in F} B_{\lam_0}(x',R/4), \]
so that
\[ F \subset A_{\lam_0, q, \eps, R}(\mu, \nu) \text{ and } \left\{ B_{\lam_0}\left(x', R/8 \right) : x' \in F \right\} \text{ consists of pairwise disjoint balls}. \]
Then by \eqref{eq: A subset G}, \eqref{eq: mu nu comparison} and definition of $G_{\lam_0, q, \eps, M}(\nu)$
\[\begin{split} \mu(A_{\lam_0, q, \eps, R}(\mu, \nu)) & \leq \sum \limits_{x' \in F} \mu(B_{\lam_0}(x',R/4) \cap A_{\lam_0, q, \eps, R}(\mu, \nu)) \leq  2^{-\eps} (R/4)^{-5\eps} \sum \limits_{x' \in F} \nu(B_{\lam_0}(x',R/4)) \\
& \leq  2^{-2\eps} (R/4)^{-5\eps} (R/8)^{-3\eps} \sum \limits_{x' \in F} \nu(B_{\lam_0}(x,R/8) \leq 2^{17\eps}R^{-8\eps} \nu(X).
\end{split}\]
This proves \eqref{eq: full set measure bound} with $M = 2^{17\eps}R^{-8\eps}$.
\end{proof}

In order to make use of the Lemma \ref{lem: A G bounds}, we need to prove that sets $A_{\lam_0, q, \eps, R}(\mu, \nu)$ have large (or at least positive) $\mu$-measure. This is achieved in the following lemma.

\begin{lem}\label{lem: V saturation}
Fix $\lam_0 \in U, q \geq  0$ and $\eps>0$. Let $\Vk_{\lam_0} \subset \Mk_\fin(X)$ be a countable set witnessing the relative dimension separability of $\Mk$ with respect to $\rho_{\lam_0}$. For $\nu \in \Vk_{\lam_0}$ set
\[ \underline{V}_{\lam_0, q, \eps}(\nu) = \{ \mu \in \Mk : \lhdim (\mu, \rho_{\lam_0}) \geq q,\ \dim(\mu || \nu, \rho_{\lam_0}) < \eps \} \]
and
\[ \overline{V}_{\lam_0, q, \eps}(\nu) = \{ \mu \in \Mk : \uhdim (\mu, \rho_{\lam_0}) \geq q,\ \dim(\mu || \nu, \rho_{\lam_0}) < \eps \}. \]
Then
\begin{equation}\label{eq: lower V cover}
\{ \mu \in \Mk : \lhdim (\mu, \rho_{\lam_0}) \geq q \} \subset \bigcup \limits_{\nu \in \Vk_{\lam_0}} \underline{V}_{\lam_0, q, \eps}(\nu)
\end{equation}
and

\begin{equation}\label{eq: upper V cover}
	\{ \mu \in \Mk : \uhdim (\mu, \rho_{\lam_0}) \geq q \} \subset \bigcup \limits_{\nu \in \Vk_{\lam_0}} \overline{V}_{\lam_0, q, \eps}(\nu).
\end{equation}
Moreover
\begin{equation}\label{eq: upper mu M limit}
\lim \limits_{R \to 0}\ \mu \left( X \setminus A_{\lam_0, q, \eps, R}(\mu, \nu) \right) = 0\text{ for every } \mu \in \underline{V}_{\lam_0, q, \eps}(\nu)
\end{equation}
and 
\begin{equation}\label{eq: lower mu M limit}
\lim \limits_{R \to 0}\ \mu \left( A_{\lam_0, q, \eps, R}(\mu, \nu) \right) > 0\text{ for every } \mu \in \overline{V}_{\lam_0, q, \eps}(\nu)
\end{equation}
\end{lem}

\begin{proof}
Equalities \eqref{eq: lower V cover} and \eqref{eq: upper V cover} follow from $\Mk \subset \bigcup \limits_{\nu \in \Vk_{\lam_0}}  \{ \mu \in \Mk : \dim(\mu || \nu, \rho_{\lam_0}) < \eps \} $, which is a consequence of the relative dimension separability. For \eqref{eq: upper mu M limit}, note that if  $\dim(\mu || \nu, \rho_{\lam_0}) < \eps$, then for $\mu$-a.e. $x \in X$, there exists $R(x)> 0$ such that for all $0 < r < R(x)$ inequalities
\begin{equation}\label{eq: rel dim small r} r^\eps \nu(B_{\lam_0}(x,r)) \leq \mu(B_{\lam_0}(x,r)) \leq r^{-\eps} \nu(B_{\lam_0}(x,r))
\end{equation}
hold.

Similarly, if $\lhdim (\mu, \rho_{\lam_0}) \geq q$, then for $\mu$-a.e. $x \in X$ there exists $R(x)$ such that
\begin{equation}\label{eq: lhdim r}  \mu(B_{\lam_0}(x,r)) \leq r^{q - \eps} \text{ for all } 0 < r < R(x)
\end{equation}
and if $\mu$ is weakly diametrically regular with respect to $\rho_{\lam_0}$, then for $\mu$-a.e. $x \in X$ there exists $R(x)>0$ such that
\begin{equation}\label{eq: wdr r} \mu(B_{\lam_0}(x,2r))  \leq r^{-\eps} \mu (B_{\lam_0}(x,r)) \text{ for all }0 < r < R(x).
\end{equation}
Combining \eqref{eq: rel dim small r}, \eqref{eq: lhdim r}, \eqref{eq: wdr r} proves \eqref{eq: upper mu M limit}. For \eqref{eq: lower mu M limit} it suffices to note that if $\uhdim \mu \geq q$ then there is a set of positive $\mu$-measure (rather than full) such that for $\mu$-a.e. $x \in X$ there exists $R(x)>0$ for which \eqref{eq: lhdim r} holds.
\end{proof}

\subsection{Theorem \ref{thm: general proj} - Hausdorff dimension}

The following proposition is the main step of the proof of the Hausdorff dimension part of Theorem \ref{thm: general proj}. Recall that for a finite Borel measure $\mu$ on $\R^d$, its $s$-energy for $s > 0$ is defined as
\[ \Ek_s(\mu) = \int \int |x-y|^{-s}d\mu(x)d\mu(y). \]

\begin{prop}\label{prop: energy bound}
Fix $L>0$. Let $U_L = \{ \lam \in U : \Lip(\Pi_\lam, \rho_\lam) \leq L \}$. Fix $\lam_0 \in U_L,\ q ,\eps >  0$ and $0 < \xi \leq 1$. There exists an open neighbourhood $U'$ of $\lam_0$ in $U_L$ with the following property: for every $\nu \in \Vk_{\lam_0}$ and every $R > 0$
\[ \int \limits_{U'} \sup \limits_{\mu \in \underline{V}_{\lam_0, q, \eps}(\nu)} \Ek_s(\Pi_\lam ( \mu|_{A_{\lam_0, q, \eps, R}(\mu,\nu)} ) )d\eta(\lam) < \infty\]
if $0< s < \min \left\{  \frac{d - \eps}{1 + \xi} - 16\eps,\  (1+\xi)\left( q - 18\eps - (d-\eps)\left(\frac{1}{1 - \xi} - \frac{1}{1+\xi}\right) - \frac{\eps}{1 - \xi} \right)\right\}$, where $\underline{V}_{\lam_0, q, \eps}(\nu)$ is as defined in Lemma \ref{lem: V saturation}.
\end{prop}

\begin{proof}
Fix $\lam_0 \in U, q > \eps >  0,\ \xi > 0, R>0, s>0$ and $\nu \in \Vk_{\lam_0}$. Let us denote for short $V = \underline{V}_{\lam_0, q, \eps}(\nu),\ A_\mu = A_{\lam_0, q, \eps, R}(\mu,\nu),\ G = G_{\lam_0, q, \eps, M}(\nu)$ and $\tilde{\mu} = \mu|_{A_\mu}$. We will use the following notation: $A \lesssim B$ means that there exists a (finite) constant $C = C(\lam_0, q, \eps, \xi, R,  s, \nu)$ such that $A \leq C B + C$ (here $A,B$ are allowed to depend on all these parameters and possibly some others). In particular, if $A \lesssim B$, then $B < \infty$ implies $A < \infty$.

Let $U' \subset U_L$ be a neighbourhood of $\lam_0$ such that $\eta(U') < \infty$ and for every $x, y \in X$ inequalities
\begin{equation}\label{eq: metric continuity} H^{-1} \rho_\lam(x,y)^{1+\xi} \leq \rho_{\lam_0}(x,y) \leq H \rho_\lam(x,y)^{1 - \xi}
\end{equation}
and
\begin{equation}\label{eq: trans ngbhd}
\eta(\{ \lam \in U' : |\Pi_\lam (x) - \Pi_\lam (y)| < \rho_\lam (x,y) r \text{ and } \rho_\lam (x,y) \geq \delta \}) \leq K \delta^{-\eps} r^{d-\eps}  \text{ for $\delta, r > 0$}
\end{equation}
hold (it exists due to assumptions \ref{it: metric compability} and \ref{it: trans}). Set $D_0 = \diam(X, \rho_{\lam_0})$ and $D := LH D_0^{\frac{1}{1+\xi}}$. Note that \eqref{eq: metric continuity} gives
\begin{equation}\label{eq: Lip bound}  \sup \limits_{\lam \in U'}|\Pi_\lam(x) - \Pi_\lam(y)| \leq LH\rho_{\lam_0}(x,y)^\frac{1}{1+\xi} \leq D.
\end{equation}
By Lemma \ref{lem: packing cover}, for each $n \geq 1$ take a finite  cover
\begin{equation}\label{eq: G cover} G \subset \bigcup \limits_{x' \in F_n} B_{\lam_0}\left(x', \frac{1}{8}D_02^{-n}\right)
\end{equation}
such that
\begin{equation}\label{eq: Fn disjoint} F_n \subset G \text{ and } \left\{ B_{\lam_0}\left(x', \frac{1}{16}D_02^{-n}\right) : x' \in F_n \right\} \text{ consists of pairwise disjoint balls}.
\end{equation}
Note that for each $\mu \in V$, we have $\lhdim \tilde{\mu} \geq \lhdim \mu \geq q > 0$, hence $\tilde{\mu}$ has no atoms. Therefore, as Lemma \ref{eq: A subset G} gives $A_{\mu} \subset G$ for $\mu \in V$, we have for each $\mu \in V,\ \lam \in U'$ and $s>0$
\begin{equation}\label{eq: energy sum n}
\begin{split} \Ek_s &   (\Pi_\lam\tilde{\mu})   = \int \limits_{G} \int \limits_{G} |\Pi_\lam(x) - \Pi_\lam(y)|^{-s} d\tilde{\mu}(x)d\tilde{\mu}(y) \\
	& = \sum \limits_{n=0}^\infty \int \limits_{G} \int \limits_{G} \mathds{1}_{\left\{D_02^{-(n+1)} < \rho_{\lam_0}(x,y) \leq D_0 2^{-n} \right\}}|\Pi_\lam(x) - \Pi_\lam(y)|^{-s} d\tilde{\mu}(x)d\tilde{\mu}(y).
\end{split}
\end{equation}
It follows from \eqref{eq: Lip bound} that
\begin{equation}\label{eq: Lip 2n}
\text{if }\ \rho_{\lam_0}(x,y) \leq D_0 2^{-n} \text{, then }\ |\Pi_\lam(x) - \Pi_\lam(y)| \leq D2^{-\frac{n}{1+\xi}}.
\end{equation}
Therefore
\begin{equation}\label{eq: diff sum m}
	\begin{split}
&\mathds{1}_{\left\{D_02^{-(n+1)} < \rho_{\lam_0}(x,y)\leq D_0 2^{-n}\right\}} |\Pi_\lam(x) - \Pi_\lam(y)|^{-s} \\
 &\qquad \leq \mathds{1}_{\left\{D_02^{-(n+1)} < \rho_{\lam_0}(x,y) \leq  D_0 2^{-n} \right\}}D^{-s} \sum \limits_{m=0}^\infty 2^{s(m+1+\frac{n}{1+\xi})} \mathds{1}_{\left\{ |\Pi_\lam(x) - \Pi_\lam(y)| \leq 2^{-m}D2^{-\frac{n}{1+\xi}} \right\}}.
 \end{split}
 \end{equation}
 Indeed, if $\rho_{\lam_0}(x,y) \leq  D_0 2^{-n} $ and $\Pi_\lam(x) \neq \Pi_\lam(y)$, then by $\eqref{eq: Lip 2n}$ there exists $m \geq 0$ such that $2^{-(m+1)}D2^{-\frac{n}{1+\xi}} < |\Pi_\lam(x) - \Pi_\lam(y)| \leq 2^{-m}D2^{-\frac{n}{1+\xi}}$, so \eqref{eq: diff sum m} follows. If $\Pi_\lam(x) = \Pi_\lam(y)$ then both sides of the inequality in \eqref{eq: diff sum m}  are infinite provided $D_02^{-(n+1)} < \rho_{\lam_0}(x,y)\leq D_0 2^{-n}$. Applying \eqref{eq: diff sum m} to \eqref{eq: energy sum n} gives

\[
\begin{split} \Ek_s &   (\Pi_\lam\tilde{\mu})   \lesssim \sum \limits_{n=0}^\infty \sum \limits_{m=0}^\infty 2^{s(m+\frac{n}{1+\xi})} \\
& \qquad \qquad \times \tilde{\mu} \otimes \tilde{\mu} \left( \left\{ (x,y) \in G^2 : D_02^{-(n+1)} < \rho_{\lam_0}(x,y) \leq D_0 2^{-n},\ |\Pi_\lam(x) - \Pi_\lam(y)| \leq 2^{-m}D2^{-\frac{n}{1+\xi}}\right\} \right).
\end{split}
\]
To bound each term in the sum, we can cover $G^2$ by products of balls from \eqref{eq: G cover} corresponding to $n+m$, obtaining
\[
\begin{split} \Ek_s    (\Pi_\lam\tilde{\mu}) &\lesssim \sum \limits_{n=0}^\infty \sum \limits_{m=0}^\infty 2^{s(m+\frac{n}{1+\xi})} \sum \limits_{x' \in F_{n+m}} \sum \limits_{y' \in F_{n+m}} \tilde{\mu}\left( B_{\lam_0}\left(x', \frac{1}{8}D_02^{-{(n+m)}}\right) \right) \tilde{\mu}\left( B_{\lam_0}\left(y', \frac{1}{8}D_0 2^{-{(n+m)}}\right) \right) \\
	& \qquad \times \mathds{1}_{\left\{ \frac{1}{2}D_02^{-(n+1)} < \rho_{\lam_0}(x', y') \leq 2D_0 2^{-n} \right\}} \times \mathds{1}_{\left\{ |\Pi_\lam(x') - \Pi_\lam(y')| \leq 2D2^{- \frac{n+m}{1 + \xi}} \right\}}.
\end{split}
\]
We have used here the observation that if $x \in B_{\lam_0}\left(x', \frac{1}{8}D_02^{-{(n+m)}}\right)$ and $y \in B_{\lam_0}\left(y', \frac{1}{8}D_02^{-{(n+m)}}\right)$ then $D_02^{-(n+1)} < \rho_{\lam_0}(x,y) \leq D_0 2^{-n}$ implies $\frac{1}{2}D_02^{-(n+1)} < \rho_{\lam_0}(x', y') \leq 2D_0 2^{-n}$, while $|\Pi_\lam(x) - \Pi_\lam(y)| \leq 2^{-m}D2^{-\frac{n}{1+\xi}}$ implies $|\Pi_\lam(x') - \Pi_\lam(y')| \leq 2^{-m}D2^{-\frac{n}{1+\xi}} + 2LH8^{-\frac{1}{1+\xi}} D_0^{\frac{1}{1+\xi}}2^{-\frac{n+m}{1 + \xi}} \leq  2D2^{- \frac{n+m}{1 + \xi}}$ by \eqref{eq: Lip bound}. Let $N \in \N$ be such that $\frac{1}{16}D_02^{-N} < R/2$. Then
\begin{equation}\label{eq: mu energy discretized}
\begin{split} \Ek_s   (\Pi_\lam&\tilde{\mu}) \\
	& \lesssim \sum \limits_{m + n > N} 2^{s(m+\frac{n}{1+\xi})} \sum \limits_{x' \in F_{n+m}} \sum \limits_{y' \in F_{n+m}} \tilde{\mu}\left( B_{\lam_0}\left(x', \frac{1}{8}D_02^{-{(n+m)}}\right) \right) \tilde{\mu}\left( B_{\lam_0}\left(y', \frac{1}{8}D_0 2^{-{(n+m)}}\right) \right) \\
	& \qquad \times \mathds{1}_{\left\{ \frac{1}{2}D_02^{-(n+1)} < \rho_{\lam_0}(x', y') \leq 2D_0 2^{-n} \right\}} \times \mathds{1}_{\left\{ |\Pi_\lam(x') - \Pi_\lam(y')| \leq 2D2^{- \frac{n+m}{1 + \xi}} \right\}},
\end{split}
\end{equation}
as by \eqref{eq: full set measure bound}
\[
\begin{split} \sum \limits_{n=0}^N & \sum \limits_{m=0}^N 2^{s(m+\frac{n}{1+\xi})} \sum \limits_{x' \in F_{n+m}} \sum \limits_{y' \in F_{n+m}} \tilde{\mu}\left( B_{\lam_0}\left(x', \frac{1}{8}D_02^{-{(n+m)}}\right) \right) \tilde{\mu}\left( B_{\lam_0}\left(y', \frac{1}{8}D_0 2^{-{(n+m)}}\right) \right) \\
	& \qquad \times \mathds{1}_{\left\{ \frac{1}{2}D_02^{-(n+1)} < \rho_{\lam_0}(x', y') \leq 2D_0 2^{-n} \right\}} \times \mathds{1}_{\left\{ |\Pi_\lam(x') - \Pi_\lam(y')| \leq 2D2^{- \frac{n+m}{1 + \xi}} \right\}}\\
	& \leq 2^{sN(1 + \frac{1}{1+\xi})} M^2\nu(X)^2 \sum \limits_{n=0}^N \sum \limits_{m=0}^N \# F_{n+m}^2 \\
	& \lesssim 1.
\end{split}
\]
Continuing from \eqref{eq: mu energy discretized}, as $F_{n+m} \subset G$, we can invoke \eqref{eq: mu nu comparison} to bound further
\[
\begin{split} \Ek_s   (\Pi_\lam\tilde{\mu})  & \lesssim \sum \limits_{m + n > N}  2^{s(m+\frac{n}{1+\xi}) + 10 \eps(n+m)} \sum \limits_{x' \in F_{n+m}} \sum \limits_{y' \in F_{n+m}} \nu\left( B_{\lam_0}\left(x', \frac{1}{8}D_02^{-{(n+m)}}\right) \right) \\
	& \qquad \times \nu\left( B_{\lam_0}\left(y', \frac{1}{8}D_0 2^{-{(n+m)}}\right) \right)  \times \mathds{1}_{\left\{ \frac{1}{2}D_02^{-(n+1)} < \rho_{\lam_0}(x', y') \leq 2D_0 2^{-n} \right\}} \times \mathds{1}_{\left\{ |\Pi_\lam(x') - \Pi_\lam(y')| \leq 2D2^{- \frac{n+m}{1 + \xi}} \right\}}.
\end{split}
\]

Note further that if $\frac{1}{2}D_02^{-(n+1)} < \rho_{\lam_0}(x', y')$, then \eqref{eq: metric continuity} implies $\rho_\lam(x', y') > Q_\xi2^{- \frac{n}{1 - \xi}}$, where $Q_\xi =  H^{-1} \left( D_0 / 4 \right)^{\frac{1}{1-\xi}}$, and so
\[\begin{split}
& \mathds{1}_{\left\{ \frac{1}{2}D_02^{-(n+1)} < \rho_{\lam_0}(x', y') \leq 2D_0 2^{-n} \right\}} \times \mathds{1}_{\left\{ |\Pi_\lam(x') - \Pi_\lam(y')| \leq 2D2^{- \frac{n+m}{1 + \xi}} \right\}} \\
& \qquad \leq \mathds{1}_{\left\{ \frac{1}{2}D_02^{-(n+1)} < \rho_{\lam_0}(x', y') \leq 2D_0 2^{-n} \right\}} \times \mathds{1}_{\left\{ |\Pi_\lam(x') - \Pi_\lam(y')| \leq C \rho_\lam(x', y') 2^{-\frac{m}{1+\xi} + n\left(\frac{1}{1-\xi} - \frac{1}{1+\xi}\right)},\  \rho_\lam(x', y') > Q_\xi2^{- \frac{n}{1 - \xi}} \right\}}
\end{split}\]
for a constant $C = C(\lam_0, \xi)$. Combing the last two bounds, which are uniform in $\mu \in V$ and $\lam \in U'$, we obtain
\[
\begin{split} \sup \limits_{\mu \in V}\ \Ek_s (\Pi_\lam\tilde{\mu}) & \lesssim \sum \limits_{m + n > N}  2^{s(m+\frac{n}{1+\xi}) + 10 \eps(n+m)} \sum \limits_{x' \in F_{n+m}} \sum \limits_{y' \in F_{n+m}} \nu\left( B_{\lam_0}\left(x', \frac{1}{8}D_02^{-{(n+m)}}\right) \right) \\
	& \qquad \times \nu\left( B_{\lam_0}\left(y', \frac{1}{8}D_0 2^{-{(n+m)}}\right) \right) \times \mathds{1}_{\left\{ \frac{1}{2}D_02^{-(n+1)} < \rho_{\lam_0}(x', y') \leq 2D_0 2^{-n} \right\}}  \\
	& \qquad\times \mathds{1}_{\left\{ |\Pi_\lam(x') - \Pi_\lam(y')| \leq C \rho_\lam(x', y') 2^{-\frac{m}{1+\xi} + n\left(\frac{1}{1-\xi} - \frac{1}{1+\xi}\right)},\  \rho_\lam(x', y') > Q_\xi2^{- \frac{n}{1 - \xi}} \right\}}.
\end{split}
\]

Integrating with respect to $d\eta(\lam)$ and using the transversality condition \eqref{eq: trans ngbhd} yields

\begin{equation}\label{eq: integral bound sums}
\begin{split}
\int \limits_{U'} & \sup \limits_{\mu \in V}\ \Ek_s (\Pi_\lam\tilde{\mu})\ d\eta(\lam)  \lesssim \sum \limits_{m + n > N}  2^{s(m+\frac{n}{1+\xi}) + 10 \eps(n+m)} \\
& \qquad \times\sum \limits_{x' \in F_{n+m}} \sum \limits_{y' \in F_{n+m}} \nu\left( B_{\lam_0}\left(x', \frac{1}{8}D_02^{-{(n+m)}}\right) \right) \nu\left( B_{\lam_0}\left(y', \frac{1}{8}D_0 2^{-{(n+m)}}\right) \right) \\
& \qquad \times \mathds{1}_{\left\{ \frac{1}{2}D_02^{-(n+1)} < \rho_{\lam_0}(x', y') \leq 2D_0 2^{-n} \right\}} \\
& \qquad \eta \left( \left\{  \lam \in U' :|\Pi_\lam(x') - \Pi_\lam(y')| \leq C \rho_\lam(x', y') 2^{-\frac{m}{1+\xi} + n\left(\frac{1}{1-\xi} - \frac{1}{1+\xi}\right)},\  \rho_\lam(x', y') > Q_\xi2^{- \frac{n}{1 - \xi}} \right\} \right) \\
& \lesssim \sum \limits_{m + n > N}  2^{s(m+\frac{n}{1+\xi}) + 10 \eps(n+m) - \frac{m(d-\eps)}{1+\xi} + n\left((d-\eps)\left(\frac{1}{1-\xi} - \frac{1}{1+\xi} \right) + \frac{\eps}{1 - \xi}\right)}  \\
& \qquad \times \sum \limits_{x' \in F_{n+m}} \sum \limits_{y' \in F_{n+m}} \nu\left( B_{\lam_0}\left(x', \frac{1}{8}D_02^{-{(n+m)}}\right) \right) \nu\left( B_{\lam_0}\left(y', \frac{1}{8}D_0 2^{-{(n+m)}}\right) \right) \\
& \qquad \times \mathds{1}_{\left\{ \frac{1}{2}D_02^{-(n+1)} < \rho_{\lam_0}(x', y') \leq 2D_0 2^{-n} \right\}}.
\end{split}
\end{equation}
To deal with the sums over $x',y'$ we recall that $F_{n+m} \subset G$ and use the definition of $G$ to obtain 
\begin{equation}\label{eq: sum x'y'}
\begin{split}
\sum \limits_{x' \in F_{n+m}} & \sum \limits_{y' \in F_{n+m}} \nu\left( B_{\lam_0}\left(x', \frac{1}{8}D_02^{-{(n+m)}}\right) \right) \nu\left( B_{\lam_0}\left(y', \frac{1}{8}D_0 2^{-{(n+m)}}\right) \right) \\
& \qquad \times \mathds{1}_{\left\{ \frac{1}{2}D_02^{-(n+1)} < \rho_{\lam_0}(x', y') \leq 2D_0 2^{-n} \right\}} \\
& \lesssim 2^{6\eps(n+m)} \sum \limits_{x' \in F_{n+m}} \sum \limits_{y' \in F_{n+m}} \nu\left( B_{\lam_0}\left(x', \frac{1}{16}D_02^{-{(n+m)}}\right) \right) \nu\left( B_{\lam_0}\left(y', \frac{1}{16}D_0 2^{-{(n+m)}}\right) \right) \\
& \qquad \times \mathds{1}_{\left\{ \frac{1}{2}D_02^{-(n+1)} < \rho_{\lam_0}(x', y') \leq 2D_0 2^{-n} \right\}}. \\
\end{split}
\end{equation}
By \eqref{eq: Fn disjoint} and the definition of $G$ we have for $n + m > N$ (recall that $N$ was chosen so that $\frac{1}{16}D_0 2^{-N} < R/2$)
\begin{equation}\label{eq: sum of y bound}
\begin{split}\sum \limits_{y' \in F_{n+m}} & \nu\left( B_{\lam_0}\left(y', \frac{1}{16}D_0 2^{-{(n+m)}}\right) \right) \mathds{1}_{\left\{ \frac{1}{2}D_02^{-(n+1)} < \rho_{\lam_0}(x', y') \leq 2D_02^{-n} \right\}} \\
	& \leq \nu\left( B_{\lam_0}(x', 4D_0 2^{-n}) \right) \\
	&  \lesssim 2^{-n(q-2\eps)},
\end{split}
\end{equation}
where the last inequality holds for large enough $n$ by the definition of $G$, while for the remaining finitely many $n$'s it holds as $\nu(X) < \infty$. Applying \eqref{eq: sum of y bound} to \eqref{eq: sum x'y'} and invoking once more disjointness of the balls in \eqref{eq: Fn disjoint} gives
\[\begin{split}
\sum \limits_{x' \in F_{n+m}} & \sum \limits_{y' \in F_{n+m}} \nu\left( B_{\lam_0}\left(x', \frac{1}{8}D_02^{-{(n+m)}}\right) \right) \nu\left( B_{\lam_0}\left(y', \frac{1}{8}D_0 2^{-{(n+m)}}\right) \right) \\
& \qquad \times \mathds{1}_{\left\{ \frac{1}{2}D_02^{-(n+1)} < \rho_{\lam_0}(x', y') \leq 2D_0 2^{-n} \right\}} \\
& \lesssim \nu(X)2^{-n(q-2\eps) + 6\eps(n+m)}. \\
\end{split}\]
Combining this with \eqref{eq: integral bound sums} gives finally
\[ \int \limits_{U'}  \sup \limits_{\mu \in V}\ \Ek_s (\Pi_\lam\tilde{\mu})\ d\eta(\lam) \lesssim  \sum \limits_{n=0}^\infty \sum \limits_{m=0}^\infty 2^{m\left(s + 16\eps - \frac{d-\eps}{1 + \xi}\right)+ n\left( \frac{s}{1+\xi} + 18\eps + (d-\eps)\left(\frac{1}{1 - \xi} - \frac{1}{1+\xi}\right) + \frac{\eps}{1 - \xi} - q  \right)}.  \]
The last sum is finite provided that
\[ s < \frac{d - \eps}{1 + \xi} - 16\eps\ \text{ and }\ s < (1+\xi)\left( q - 18\eps - (d-\eps)\left(\frac{1}{1 - \xi} - \frac{1}{1+\xi} - \frac{\eps}{1 - \xi}\right) \right). \]
\end{proof}

\begin{proof}[{\bf Proof of point \eqref{it: theorem general hdim} of Theorem \ref{thm: general proj}}]
By taking a countable intersection over $\lam$, it suffices to prove that for every $L \geq 1$, the conclusion of the theorem holds for $\eta$-a.e. $\lam \in U_L = \{ \lam \in U : \Lip(\Pi_\lam, \rho_\lam) \leq L \}$.

It is well known that for a Borel measure $\mu$ on $\R^n$, if $\Ek_s(\mu) < \infty$, then $\lhdim \mu \geq s$, see \cite[Theorem 4.13]{FalconerFoundations}. Therefore, a consequence of Proposition \ref{prop: energy bound} is that for fixed $q, \eps > 0$ and $0 < \xi\leq 1$ there exists an open neighbourhood $U'$ of $\lam_0$ in $U_L$  such for every $\nu \in \Vk_{\lam_0}$ and every $R > 0$ we have that for $\eta$-a.e. $\lam \in U'$ inequality
\begin{equation}\label{eq: lhdim bound mu A} \lhdim \Pi_\lam ( \mu|_{A_{\lam_0, q, \eps, R}(\mu,\nu)} )  \geq \min \left\{  \frac{d - \eps}{1 + \xi} - 16\eps,\  (1+\xi)\left( q - 18\eps - (d-\eps)\left(\frac{1}{1 - \xi} - \frac{1}{1+\xi}\right) - \frac{\eps}{1 - \xi} \right)\right\}
\end{equation}
holds for every $\mu \in \underline{V}_{\lam_0, q, \eps}(\nu)$.

Point \eqref{it: theorem general hdim} of Theorem \ref{thm: general proj} follows from \eqref{eq: lhdim bound mu A} by invoking Lemma \ref{lem: V saturation}, letting $q,\eps,\xi, R \to 0$ and taking countable intersections over the parameter space. Below we explain this more precisely.

Letting $R \to 0$ in \eqref{eq: lhdim bound mu A} and taking countable intersection over $\lam$ we see by \eqref{eq: upper mu M limit} that for $\eta$-a.e. $\lam \in U'$ inequality
\begin{equation}\label{eq: lhdim bound mu} \lhdim \Pi_\lam  \mu  \geq \min \left\{  \frac{d - \eps}{1 + \xi} - 16\eps,\  (1+\xi)\left( q - 18\eps - (d-\eps)\left(\frac{1}{1 - \xi} - \frac{1}{1+\xi}\right) - \frac{\eps}{1 - \xi} \right)\right\}.
\end{equation}
holds for every $\mu \in \underline{V}_{\lam_0, q, \eps}(\nu)$. By \eqref{eq: lower V cover}, as $\Vk_{\lam_0}$ is countable, we have that for $\mu$-a.e. $\lam \in U'$, inequality \eqref{eq: lhdim bound mu} holds for every $\mu \in \Mk$ with $\lhdim( \mu, \rho_{\lam_0}) \geq q > 0$. We can assume by \ref{it: metric compability} that $U'$ is such that $H^{-1} \rho_\lam(x,y)^{1+\xi} \leq \rho_{\lam_0}(x,y) \leq H \rho_\lam(x,y)^{1 - \xi}$ holds for $\lam \in U$ and therefore $\lhdim( \mu, \rho_{\lam_0}) \geq \frac{\lhdim( \mu, \rho_{\lam})}{1 + \xi}$ for all $\lam \in U'$ and $\mu \in \Mk$ (since $B_{\lam_0}(x,r) \subset B_\lam(x, Hr^\frac{1}{1+\xi})$). We conclude that for $\eta$-a.e. $\lam \in U'$, inequality \eqref{eq: lhdim bound mu} holds for every $\mu \in \Mk$ satisfying $\lhdim( \mu, \rho_{\lam}) \geq (1+\xi)q > 0$. As $U_L$ is Lindel\"of (since $U$ is hereditary Lindel\"of) we can find a countable cover $\{ U'_i \}_{i=1}^\infty$ of $U_L$ such that for $\eta$-a.e. $\lam \in U'_i$ inequality \eqref{eq: lhdim bound mu} holds for every $\mu \in \Mk$ with $\lhdim( \mu, \rho_{\lam}) \geq (1+\xi)q > 0$ and hence the same is true for $\eta$-a.e $\lam \in U_L$. Finally, taking countable intersections over $q, \eps, \xi \in \mathbb{Q} \cap (0,\infty)$ we conclude that for $\eta$-a.e. $\lam \in U_L$ we have
\[  \lhdim \Pi_\lam  \mu  \geq \min \left\{d, \lhdim \mu \right\} \text{ for every } \mu \in \Mk \text{ such that } \lhdim (\mu, \rho_\lam) > 0.\]
The proof of the Theorem for the lower Hausdorff dimension is finished upon noting that $\lhdim \Pi_\lam  \mu \leq \lhdim (\mu, \rho_\lam)$ as $\Pi_\lam$ is Lipschitz in $\rho_\lam$ (so in particular $\lhdim \Pi_\lam = 0$ if $\lhdim \mu = 0$) and $\lhdim \Pi_\lam \mu \leq d$ as $\Pi_\lam \mu$ is a measure on $\R^d$.

The case of the upper Hausdorff dimension can be treated in exactly the same way, with the use of  \eqref{eq: upper V cover} and \eqref{eq: lower mu M limit} instead of \eqref{eq: lower V cover} and \eqref{eq: upper mu M limit}, respectively.
\end{proof}

\subsection{Theorem \ref{thm: general proj} - Assouad dimension and H\"older regularity}

The remainder of this section is devoted to the proof of point \eqref{it: theorem general adim} of Theorem \ref{thm: general proj}. We begin with a lemma which combines the transversality assumption \ref{it: trans} with  covering bounds coming from the Assouad dimension.

\begin{lem}\label{lem: adim trans}
Fix $L > 0$ and set $U_L = \{ \lam \in U : \Lip(\Pi_\lam, \rho_\lam) \leq L \}$. For every $\lam_0 \in U_L,\ \theta>0$ there exists an open neighbourhood $U'$ of $\lam_0 \in U_L$ and a constant $D=D(\lam_0, L, \theta)$ such that for every $x\ \in X, 0 < r \leq \delta \leq 1$ inequality
\begin{equation}\label{it: adim trans}
\eta \left( \left\{ \lam \in U' : \underset{y \in X}{\exists}\ |\Pi_\lam(x) - \Pi_\lam(y)| \leq r,\ \rho_{\lam_0}(x,y) \geq \delta  \right\} \right) \leq D \left( \frac{r}{\delta} \right)^{d - \adim(X,\rho_{\lam_0}) - \theta}\delta^{-\theta}
\end{equation}
holds provided that $\adim(X, \rho_{\lam_0}) < d$.
\end{lem}

\begin{proof}
Fix $\xi > 0$ (we will specify it later in terms of $\theta$ and $\theta'$) and let $U' \subset U_L$ be a neighbourhood of $\lam_0$ such that $\eta(U') < \infty$ and for every $x, y \in X$ inequalities
\begin{equation}\label{eq: metric continuity 2} H^{-1} \rho_\lam(x,y)^{1+\xi} \leq \rho_{\lam_0}(x,y) \leq H \rho_\lam(x,y)^{1 - \xi}
\end{equation}
and
\begin{equation}\label{eq: trans ngbhd 2}
	\eta(\{ \lam \in U' : |\Pi_\lam (x) - \Pi_\lam (y)| < \rho_\lam (x,y) r \text{ and } \rho_\lam (x,y) \geq \delta \}) \leq K \delta^{-\xi} r^{d-\xi}  \text{ for $\delta, r > 0$}
\end{equation}
hold. Fix $x \in X$. Given $i \geq 0$ take a cover
\[ \left\{ y \in X : 2^i \delta \leq \rho_{\lam_0}(x,y) < 2^{i+1}\delta \right\} \subset \bigcup \limits_{m=1}^{N_i} B_{\lam_0}(y_{i,m}, r)\]
such that
\begin{equation}\label{eq: assouad bound} N_i \leq C_{\xi}\left( \frac{2^{i+1}\delta}{r} \right)^{\adim(X, \rho_{\lam_0}) + \xi}\ \text{ and }\ 2^i\delta \leq \rho_{\lam_0}(x,y_{i,m}) < 2^{i+1}\delta.
\end{equation}
It exists due to the definition of the Assouad dimension (note that $C_\xi$ does not depend on $x$). By \eqref{eq: metric continuity 2}  we have for $\lam \in U',\ i \geq 0,\ 1 \leq m \leq N_i$
\begin{equation}\label{eq: metrics bound}
B_{\lam_0}(y_{i,m}, r) \subset B_\lam(y_{i,m}, Hr^{\frac{1}{1+\xi}}) \text{ and }  H^{-1} 2^{\frac{i}{1-\xi}}\delta^{\frac{1}{1-\xi}} \leq \rho_{\lam}(x,y_{i,m}) < H2^{\frac{i+1}{1+\xi}}\delta^{\frac{1}{1+\xi}}.
\end{equation}
This gives
\begin{equation}\label{eq: eta r delta bound 1}
\begin{split}
	\eta & \left( \left\{ \lam \in U' : \underset{y \in X}{\exists}\ |\Pi_\lam(x) - \Pi_\lam(y)| \leq r,\ \rho_{\lam_0}(x,y) \geq \delta  \right\} \right) \\
	& \leq \sum \limits_{i=0}^\infty\eta  \left( \left\{ \lam \in U' : \underset{y \in X}{\exists}\ |\Pi_\lam(x) - \Pi_\lam(y)| \leq r,\ 2^{i}\delta \leq \rho_{\lam_0}(x,y) < 2^{i+1}\delta  \right\} \right) \\
	& \leq \sum \limits_{i=0}^\infty \sum \limits_{m=1}^{N_i} \eta \left( \left\{ \lam \in U' : \underset{y \in B_{\lam_0}(y_{i,m}, r)}{\exists}\ |\Pi_\lam(x) - \Pi_\lam(y)| \leq r  \right\} \right) \\
	& \leq  \sum \limits_{i=0}^\infty \sum \limits_{m=1}^{N_i} \eta \left( \left\{ \lam \in U' : \underset{y \in  B_\lam(y_{i,m}, Hr^{\frac{1}{1+\xi}}) }{\exists}\ |\Pi_\lam(x) - \Pi_\lam(y)| \leq r  \right\} \right) \\
\end{split}
\end{equation}
If $\lam \in U'$ and $y \in B_\lam(y_{i,m}, Hr^{\frac{1}{1+\xi}})$ is such that $|\Pi_\lam(x) - \Pi_\lam(y)| \leq r$, then $|\Pi_\lam(x) - \Pi_\lam(y_{i,m})| \leq |\Pi_\lam(x) - \Pi_\lam(y)| + |\Pi_\lam(y) - \Pi_\lam(y_{i,m})| \leq r + LHr^{\frac{1}{1+\xi}} \leq Mr^{\frac{1}{1+\xi}}$ for a constant $M = LH+1$. 
Using this together with \eqref{eq: metrics bound}, we can continue \eqref{eq: eta r delta bound 1} to obtain
\[\begin{split}
\eta & \left( \left\{ \lam \in U' : \underset{y \in X}{\exists}\ |\Pi_\lam(x) - \Pi_\lam(y)| \leq r,\ \rho_{\lam_0}(x,y) \geq \delta  \right\} \right) \\
& \leq  \sum \limits_{i=0}^\infty \sum \limits_{m=1}^{N_i} \eta \left( \left\{ \lam \in U' :  |\Pi_\lam(x) - \Pi_\lam(y_{i,m})| \leq Mr^{\frac{1}{1+\xi}},\ \rho_{\lam}(x,y_{i,m}) \geq  H^{-1}2^{\frac{i}{1-\xi}}\delta^{\frac{1}{1-\xi}}   \right\} \right) \\
& \leq  \sum \limits_{i=0}^\infty \sum \limits_{m=1}^{N_i} \eta \left( \left\{ \lam \in U' :  |\Pi_\lam(x) - \Pi_\lam(y_{i,m})| \leq \rho_{\lam}(x,y_{i,m}) MHr^{\frac{1}{1+\xi}} 2^{\frac{-i}{1-\xi}}\delta^{\frac{-1}{1-\xi}},\ \rho_{\lam}(x,y_{i,m}) \geq  H^{-1}2^{\frac{i}{1-\xi}}\delta^{\frac{1}{1-\xi}}   \right\} \right).
\end{split} \]
Applying \eqref{eq: trans ngbhd 2} and \eqref{eq: assouad bound} gives
\[\begin{split}
	\eta & \left( \left\{ \lam \in U' : \underset{y \in X}{\exists}\ |\Pi_\lam(x) - \Pi_\lam(y)| \leq r,\ \rho_{\lam_0}(x,y) \geq \delta  \right\} \right) \\
	& \leq  K  (MH)^{d - \xi}r^{\frac{d - \xi}{1+\xi}}H^\xi \delta^{\frac{-d}{1-\xi}} \sum \limits_{i=0}^\infty N_i   2^{\frac{-id}{1-\xi}} \\
	& \leq C r^{\frac{d - \xi}{1+\xi} - \adim(X, \rho_{\lam_0}) - \xi } \delta^{-\frac{d}{1-\xi} + \adim(X, \rho_{\lam_0}) + \xi } \sum \limits_{i=0}^\infty 2^{i\left(\adim(X, \rho_{\lam_0}) + \xi - \frac{d}{1-\xi}\right)},
\end{split} \]
for some constant $C= C(\lam_0, L, \xi)$. If now $\xi > 0$ was chosen small enough to guarantee $\adim(X, \rho_{\lam_0}) + \xi - \frac{d}{1-\xi} < 0$ (recall that we consider only the case $\adim(X, \rho_{\lam_0}) < d$), then the above sum converges, giving
\[\begin{split}
	\eta & \left( \left\{ \lam \in U' : \underset{y \in X}{\exists}\ |\Pi_\lam(x) - \Pi_\lam(y)| \leq r,\ \rho_{\lam_0}(x,y) \geq \delta  \right\} \right) \\
	& \leq D \left( \frac{r}{\delta} \right)^{\frac{d - \xi}{1+\xi} - \adim(X, \rho_{\lam_0}) -\xi } \delta^{-\left(\frac{d}{1-\xi} - \frac{d-\xi}{1+\xi} \right) },
\end{split} \]
for some constant $D = D(\lam_0, L, \xi)$. Finally, if $\xi>0$ was chosen small enough to satisfy $\frac{d - \xi}{1+\xi} - \adim(X, \rho_{\lam_0}) -  \xi \geq  d -\adim(X, \rho_{\lam_0}) - \theta$ and $\frac{d}{1-\xi} - \frac{d-\xi}{1+\xi} \leq \theta$, then \eqref{it: adim trans} holds.
\end{proof}

\begin{proof}[{\bf Proof of point \eqref{it: theorem general adim} of Theorem \ref{thm: general proj}}]
Fix $\alpha \in (0,1),\ L \geq 1$ and set 
\[ U_L = \{ \lam \in U : \Lip(\Pi_\lam, \rho_\lam) \leq L\ \text{ and } \adim(X, \rho_\lam) < d \}.\]
As $U$ is hereditary Lindel\"of, by taking countable intersections in the parameter space $U$ it suffices to prove that every $\lam_0 \in U_L$ has an open neighbourhood $U'$ in $U_L$ such that for $\eta$-a.e. $\lam \in U'$, every $\mu \in \Mk$ has the property that for $\mu$-a.e. $x \in X$ there exists $C$ such that
\begin{equation}\label{eq: holder ineq}  \rho_\lam(x,y) \leq C  |\Pi_\lam(x) - \Pi_\lam(y) |^\alpha \text{ for every } y \in X.
\end{equation}
For that, by Lemma \ref{lem: V saturation} with $q=0$ (and again taking countable intersections over $\lam$) it suffices to show that for $\eps>0$ small enough one has for every $\nu \in \Vk_{\lam_0}$ and $R>0$
\begin{equation}\label{eq: holder goal}
\lim \limits_{C \to \infty}\ \int \limits_{U'} \sup \limits_{\mu \in  \underline{V}_{\lam_0, 0, \eps}(\nu)} \mu|_{A_{\lam_0, 0, \eps, R}(\mu, \nu)} \left( E_{C,\lam} \right) d\eta(\lam) = 0,
\end{equation}
where
\[ E_{C,\lam} =  \left\{ x \in X : \underset{y \in X}{\exists}\ \rho_\lam(x,y) > C|\Pi_\lam(x) - \Pi_\lam(y)|^\alpha \right\}. \]
Indeed, if \eqref{eq: holder goal} holds, then (note that $\mu|_{A_{\lam_0, 0, \eps, R}(\mu, \nu)} \left( E_{C,\lam} \right)$ is decreasing in $C$)
\[ \lim \limits_{C \to \infty}\ \sup \limits_{\mu \in  \underline{V}_{\lam_0, 0, \eps}(\nu)} \mu|_{A_{\lam_0, 0, \eps, R}(\mu, \nu)} \left( E_{C,\lam} \right) = 0 \text{ for } \eta\text{-a.e. } \lam \in U'  \]
and hence
\[ \text{ for } \eta\text{-a.e. } \lam \in U' ,\ \lim \limits_{C \to \infty} \mu|_{A_{\lam_0, 0, \eps, R}(\mu, \nu)} \left( E_{C,\lam} \right) = 0 \text{ for every } \mu \in  \underline{V}_{\lam_0, 0, \eps}(\nu).\]
This shows that for $\eta$-a.e. $\lam \in U'$ and every $\mu \in  \underline{V}_{\lam_0, 0, \eps}(\nu)$, for $\mu$-a.e. every $x \in A_{\lam_0, 0, \eps, R}(\mu, \nu)$  there exists $C$ such that \eqref{eq: holder ineq} holds. By \eqref{eq: upper mu M limit} the above holds then for $\mu$-a.e. $x \in X$ and by \eqref{eq: lower V cover} we have $\Mk \subset \bigcup \limits_{\nu \in \Vk_{\lam_0}} \underline{V}_{\lam_0, 0, \eps}(\nu)$ with the sum being countable, so we can extend it further to every $\mu \in \Mk$.

We shall now prove that for fixed $\lam_0 \in U_L$ there exists a neighbourhood $U'$ of $\lam_0$ in $U_L$ such that \eqref{eq: holder goal} holds provided that $\eps>0$ is small enough. For simplicity denote $V = \underline{V}_{\lam_0, 0, \eps}(\nu),\ A_\mu = A_{\lam_0, 0, \eps, R}(\mu,\nu),\ G = G_{\lam_0, 0, \eps, R}(\nu)$ and $\tilde{\mu} = \mu|_{A_\mu}$.  For $\xi>0$ (we will specify later how small $\xi$ has to be) let $U'$ be a neighbourhood of $\lam_0$ such that for $\lam \in U'$
\begin{equation}\label{eq: metric continuity 3} H^{-1} \rho_\lam(x,y)^{1+\xi} \leq \rho_{\lam_0}(x,y) \leq H \rho_\lam(x,y)^{1 - \xi} \text{ for every } x,y \in X.
\end{equation}
Set $D_0 = \diam(X, \rho_{\lam_0})$ and $D := LH D_0^{\frac{1}{1+\xi}}$, so that $\sup \limits_{\lam \in U'}\ \sup \limits_{x,y \in X} |\Pi_\lam(x) - \Pi_\lam(y)| \leq D$. By Lemma \ref{lem: packing cover}, for $i \geq 0$ take a finite cover
\begin{equation}\label{eq: G cover adim} G \subset \bigcup \limits_{x' \in F_i} B_{\lam_0}(x', D_0 2^{-i})
\end{equation}
such that
\begin{equation}\label{eq: Fi properties} F_i \subset G \text{ and } \left\{ B_{\lam_0}(x', \frac{1}{2}D_0 2^{-i}) : x' \in F_i \right\} \text{ consists of pairwise disjoint balls}.
\end{equation}
Let
\[ R_{C, i}= \left\{ (x,\lam) \in G \times U' :  \underset{y \in X}{\exists}\ D2^{-(i+1)} < |\Pi_\lam(x) - \Pi_\lam(y)| \leq D2^{-i},\ \rho_\lam(x,y) > C|\Pi_\lam(x) - \Pi_\lam(y)|^\alpha\right\}.  \]
Note that if $x \in B_{\lam_0}(x', D_02^{-i})$ and $(x,\lam) \in R_{C,i}$, then there exists $y \in X$ such that 
\[|\Pi_\lam(x') - \Pi_\lam(y)| \leq |\Pi_\lam(x') - \Pi_\lam(x)| +  |\Pi_\lam(x) - \Pi_\lam(y)| \leq LH D_0^{\frac{1}{1+\xi}}2^{-\frac{i}{1+\xi}} +  D2^{-i} \leq 2D2^{-\frac{i}{1+\xi}}\] 
and
\[\rho_{\lam_0}(x', y) \geq \rho_{\lam_0}(x, y)  - \rho_{\lam_0}(x, x')  \geq H^{-1}C^{1+\xi}D^{1+\xi}2^{-\alpha(1+\xi)(i+1)} - D_02^{-i} \geq aC2^{-\alpha(1+\xi) i}\]
for some constant $a = a(\lam_0, L, \xi)>0$ provided that $C$ is large enough and $\xi>0$ is small enough to guarantee $\alpha(1+\xi) < 1$. Therefore, for such $C$ and $\xi$ we have by \eqref{eq: G cover adim}
\[ R_{C,i} \subset \bigcup \limits_{x' \in F_i} B_{\lam_0}(x', D_02^{-i}) \times Q_{C,i}(x'), \]
where
\[ Q_{C,i}(x') = \{ \lam \in U' : |\Pi_\lam(x') - \Pi_\lam(y)| \leq 2D2^{-\frac{i}{1+\xi}},\ \rho_{\lam_0}(x', y) \geq  aC2^{-\alpha (1+\xi) i}  \}.\]
This gives
\begin{equation}\label{eq: EC sum bound}
\begin{split}\int \limits_{U'} \sup \limits_{\mu \in  V} \tilde{\mu}\left( E_{C,\lam} \right) d\eta(\lam) &\leq \sum \limits_{i=0}^\infty \int \limits_{U'} \sup \limits_{\mu \in  V} \tilde{\mu}\left( \left\{ x \in G : (x,\lam) \in R_{C,i} \right\} \right) d\eta(\lam) \\
& \leq  \sum \limits_{i=0}^\infty \int \limits_{U'} \sup \limits_{\mu \in  V} \sum \limits_{x' \in F_i} \tilde{\mu}(B_{\lam_0}(x', D_02^{-i}))\mathds{1}_{Q_{C,i}(x')}(\lam)d\eta(\lam).
\end{split}
\end{equation}
Below we use the following notation: $A \lesssim B$ means that there exists a constant $C = C(\lam_0, \eps, \xi, R, L, \nu)$ such that $A \leq C B$. Let now $N$ be such that $\frac{1}{2}D_02^{-N} < R/2$ and note that applying \eqref{eq: full set measure bound} gives
\begin{equation}\label{eq: holder bound finite sum}
\begin{split}\lim \limits_{C \to \infty}&\ \sum \limits_{i=0}^N \int \limits_{U'} \sup \limits_{\mu \in  V} \sum \limits_{x' \in F_i} \tilde{\mu}(B_{\lam_0}(x', D_02^{-i}))\mathds{1}_{Q_{C,i}(x')}(\lam)d\eta(\lam) \\
& \leq \lim \limits_{C \to \infty}\ \sum \limits_{i=0}^N \int \limits_{U'} \sup \limits_{\mu \in  V} \sum \limits_{x' \in F_i} \tilde{\mu}(B_{\lam_0}(x', D_02^{-i}))\mathds{1}_{Q_{C,i}(x')}(\lam)d\eta(\lam) \\
& \leq \lim \limits_{C \to \infty}\ M\nu(X) \sum \limits_{i=0}^N \sum \limits_{x' \in F_i} \eta(Q_{C,i}(x'))\\
& = 0,
\end{split}
\end{equation}
since for each $i \in \N$, one has $Q_{C,i}(x') = \emptyset$ for $C$ large enough (as $\diam(X, \rho_{\lam_0}) < \infty$). Therefore, by \eqref{eq: EC sum bound}, in order to prove \eqref{eq: holder goal} it suffices to show
\begin{equation}\label{eq: holder goal 2}
\lim \limits_{C \to \infty}\ \sum \limits_{i=N}^\infty\ \int \limits_{U'} \sup \limits_{\mu \in  V} \sum \limits_{x' \in F_i} \tilde{\mu}(B_{\lam_0}(x', D_02^{-i}))\mathds{1}_{Q_{C,i}(x')}(\lam)d\eta(\lam) = 0.
\end{equation}
As $F_i \subset G$, we can invoke \eqref{eq: mu nu comparison} and recall the definition of $G$ to obtain (we use here  $\frac{1}{2}D_02^{-N} < R/2$)
\[
\begin{split}\sum \limits_{i=N}^\infty\ & \int \limits_{U'} \sup \limits_{\mu \in  V} \sum \limits_{x' \in F_i} \tilde{\mu}(B_{\lam_0}(x', D_02^{-i}))\mathds{1}_{Q_{C,i}(x')}(\lam)d\eta(\lam)\\
	 & \lesssim \sum \limits_{i=N}^\infty 2^{5 \eps i} \int \limits_{U'} \sum \limits_{x' \in F_i} \nu(B_{\lam_0}(x', D_02^{-i}))\mathds{1}_{Q_{C,i}(x')}(\lam)d\eta(\lam) \\
& \lesssim \sum \limits_{i=N}^\infty 2^{8 \eps i} \int \limits_{U'} \sum \limits_{x' \in F_i} \nu(B_{\lam_0}(x', \frac{1}{2}D_02^{-i}))\mathds{1}_{Q_{C,i}(x')}(\lam)d\eta(\lam) \\
& = \sum \limits_{i=N}^\infty 2^{8 \eps i} \sum \limits_{x' \in F_i} \nu(B_{\lam_0}(x', \frac{1}{2}D_02^{-i}))\eta(Q_{C,i}(x')).
\end{split}
 \]
 Applying Lemma \ref{lem: adim trans} with $\theta = \theta(\lam_0, \eps, \xi, \alpha)$ to be specified later gives
 \[
 \begin{split} \sum \limits_{i=N}^\infty\ & \int \limits_{U'} \sup \limits_{\mu \in  V} \sum \limits_{x' \in F_i} \tilde{\mu}(B_{\lam_0}(x', D_02^{-i}))\mathds{1}_{Q_{C,i}(x')}(\lam)d\eta(\lam)\\
 & \lesssim C^{\adim(X, \rho_{\lam_0}) - d} \sum \limits_{i=0}^\infty 2^{8 \eps i} \sum \limits_{x' \in F_i} \nu(B_{\lam_0}(x', \frac{1}{2}D_02^{-i})) 2^{i(\alpha (1+\xi) - \frac{1}{1+\xi})(d - \adim(X, \rho_{\lam_0}) - \theta)}2^{\alpha (1+\xi)\theta i} \\
 & = C^{\adim(X, \rho_{\lam_0}) - d} \sum \limits_{i=0}^\infty 2^{i\left(8\eps + (\alpha (1+\xi) - \frac{1}{1+\xi})(d - \adim(X, \rho_{\lam_0}) - \theta) + \alpha\theta(1+\xi) \right)} \sum \limits_{x' \in F_i} \nu(B_{\lam_0}(x', \frac{1}{2}D_02^{-i})) \\
 & \lesssim C^{\adim(X, \rho_{\lam_0}) - d} \sum \limits_{i=0}^\infty 2^{i\left(8\eps + (\alpha (1+\xi) - \frac{1}{1+\xi})(d - \adim(X, \rho_{\lam_0}) - \theta) + \alpha\theta(1+\xi) \right)},
 \end{split}
 \]
 where the last step uses the disjointness in \eqref{eq: Fi properties}. Since $\alpha < 1$ and $\adim(X,\rho_{\lam_0}) < d$, the last sum converges provided that $\xi, \theta, \eps>0$ are chosen small enough. This choice establishes a neighbourhood $U'$ of $\lam_0$ such that
 \[\sum \limits_{i=N}^\infty\  \int \limits_{U'} \sup \limits_{\mu \in  V} \sum \limits_{x' \in F_i} \tilde{\mu}(B_{\lam_0}(x', D_02^{-i}))\mathds{1}_{Q_{C,i}(x')}(\lam)d\eta(\lam) \lesssim  C^{\adim(X, \rho_{\lam_0}) - d}\]
 and hence \eqref{eq: holder goal 2} holds as $\adim(X, \rho_{\lam_0}) < d$. Together with \eqref{eq: holder bound finite sum}, this establishes \eqref{eq: holder goal} and concludes the proof.
\end{proof}

\section{Iterated function systems and measures on symbolic spaces}

In this section we develop tools needed for applying Theorem \ref{thm: general proj} in the setting of symbolic dynamics related to IFS. In particular we prove Propositions \ref{prop: rel dim sep IFS} and \ref{prop: ede holder}.

\subsection{Relative dimension separability in symbolic spaces}

The goal of this subsection is to prove Proposition \ref{prop: rel dim sep IFS}. The proof is based on approximating ergodic measures with Markov measures in relative entropy, a technique well-known in information theory. Let us introduce some notation and recall useful results. Most of this exposition follows \cite{GrayEntropyBook}, but note that we use a different notation.

Let $\Ak$ be a finite set and let $\Sigma = \Ak^\N$ be the corresponding symbolic space.  We endow $\Sigma$ with the product topology. We will denote by $\Mk_{\sigma}(\Sigma)$ the set of all shift-invariant Borel probability measures on $\Sigma$ and by  $\Ek_\sig(\Sigma) \subset \Mk_{\sigma}(\Sigma)$ the set of all ergodic measures. For $\mu \in \MM_{\sigma}(\Sigma)$, we will denote by $\mu|_n$ the distribution of $\mu$ on words of length $n$, i.e. $\mu|_n \in \MM(\Ak^n)$ is given by $\mu|_n(\{\om\}) = \mu([\om])$ for $\om \in \Ak^n$.

\begin{defn}
	Measure $\nu \in \MM_\sigma(\Sigma)$ is called a \textbf{$k$-step Markov measure} for $k \in \N$ if for every $n \geq k$ and every $\om = (\om_1, \ldots, \om_n) \in \mA^n$ it satisfies
	\begin{equation}\label{eq:k-step Markov def} \nu( [\om_1, \ldots, \om_n] ) = \nu([\om_1, \ldots, \om_k]) \prod \limits_{j = 1}^{n-k} P_\nu(\om_{j + k} | \om_j, \ldots, \om_{j+k- 1} ),
	\end{equation}
	where $P_\nu$ is the transition kernel given by
	\[ P_\nu(\tau | \tau_1, \ldots, \tau_k) = \begin{cases} \frac{\nu([\tau_1, \ldots, \tau_k, \tau])}{\nu([\tau_1, \ldots, \tau_k])} & \text{ if } \nu([\tau_1, \ldots, \tau_k]) > 0 \\
		0 & \text{ if } \nu([\tau_1, \ldots, \tau_k]) = 0 \end{cases}\ \ \text{ for } \tau, \tau_1, \ldots \tau_k \in \mA. \]
\end{defn}
Note that a $k$-step Markov measure is uniquely determined by its stationary distribution $\nu|_k$ and its transition kernel $P_\nu$ (so it is in fact uniquely determined by its $k+1$-dimensional distribution $\nu|_{k+1}$). Entropy of a $k$-step Markov measure is given by the following formula (see e.g. \cite[Lemma 3.16]{GrayEntropyBook}):

\begin{equation}\label{eq:markov entropy}
	h(\nu) = - \sum \limits_{\om \in \mA^{k+1}}\nu([\om]) \log P_{\nu}(\om_{k+1} | \om_1, \ldots, \om_k).
\end{equation}
There is also the following criterion for ergodicity of a $k$-step Markov measure:
\begin{lem}\label{lem:markov ergodic}
	A $k$-step Markov measure $\nu$ is ergodic if and only if for every $\om, \tau \in \Ak^k$ with $\nu([\om])>0, \nu([\tau])>0$, there exists $u \in \Sigma_*$ with $\nu([u])>0$ such that $\om$ is a prefix of $u$ and $\tau$ is a suffix of $u$.
\end{lem}
\begin{proof}
	For $1$-step Markov measures this can be found e.g. in \cite[Theorem 1.13]{WaltersErgodicBook}. Statement for a $k$-step Markov measure follows by noting that it is isomorphic to a $1$-step Markov measure over alphabet $\mA^k$. Also, statement in \cite{WaltersErgodicBook} requires the stationary distribution $\nu|_k$ to be strictly positive, hence in general case one has to consider only the states with positive measure.
\end{proof}

\begin{defn}
	Let $\mu, \nu \in \MM_\sig(\Sigma)$ be such that $\mu|_n \ll \nu|_n$ for every $n \in \N$. The \textbf{relative entropy} of $\mu$ with respect to $\nu$ is defined as
	\[ h(\mu||\nu) = \lim \limits_{n \to \infty} \frac{1}{n} \sum \limits_{\om \in \mA^n} \mu([\om]) \log \frac{\mu([\om])}{\nu([\om])}, \]
	whenever the limit exists (we use the standard convention $0 \log \frac{0}{0} = 0 \log 0 = 0$).
\end{defn}
Whenever the relative entropy exists, it satisfies $h(\mu||\nu) \geq 0$, see \cite[Lemma 3.1]{GrayEntropyBook}. It may fail to exists for general shift-invariant measures, but it is guaranteed to exists if the reference measure $\nu$ is a $k$-step Markov measure:

\begin{lem}[{\cite[Lemma 3.10]{GrayEntropyBook}}]\label{lem:relative entropy markov formula}
	Let $\mu, \nu \in \MM_\sig(\Sigma)$ be such that $\mu|_n \ll \nu|_n$ for every $n \in \N$ and assume that $\nu$ is a $k$-step Markov measure. Then $h(\mu||\nu)$ is well defined and satisfies
	\[ h(\mu||\nu) = - h(\mu) - \sum \limits_{\om \in \mA^{k+1}} \mu([\om]) \log P_\nu(\om_{k+1}|\om_1, \ldots, \om_k). \]
\end{lem}

We will make use of the following ergodic theorem for relative entropy.

\begin{thm}[{\cite[Theorem 11.1]{GrayEntropyBook}}]\label{thm:relative entropy ergodic thm}
	Let $\mu \in \Ek_\sig(\Sigma)$ and $\nu \in \MM_\sig(\Sigma)$ be such that $\mu|_n \ll \nu|_n$ for every $n \in \N$ and assume that $\nu$ is a $k$-step Markov measure. Then
	\[ \lim \limits_{n \to \infty} \frac{1}{n}\log \frac{\mu([\om|_n])}{\nu([\om|_n])} = h(\mu||\nu) \text{ for } \mu\text{-a.e. } \om \in \Sigma.\]
\end{thm}

We will be interested in approximating ergodic measures by Markov measures.

\begin{defn}
	For $\mu \in \MM_\sig(\Sigma)$ and $k \in \N$, we define the \textbf{$k$-th Markov approximation of $\mu$} to be  the $k$-step Markov measure $\mu^{(k)}$ with initial distribution $\mu^{(k)}|_k = \mu|_k$
	and transition kernel $P_{\mu^{(k)}}$ given by
	\[ P_{\mu^{(k)}}(\tau | \tau_1, \ldots, \tau_k) = \begin{cases} \frac{\mu([\tau_1, \ldots, \tau_k, \tau])}{\mu([\tau_1, \ldots, \tau_k])} & \text{ if } \mu([\tau_1, \ldots, \tau_k]) > 0 \\
		0 & \text{ if } \mu([\tau_1, \ldots, \tau_k]) = 0 \end{cases} \text{ for } \tau, \tau_1, \ldots \tau_k \in \mA. \]
\end{defn}

It is straightforward to check that $\mu^{(k)}$ is well defined and it is a shift-invariant measure (as we have assumed $\mu$ to be shift-invariant). Moreover, Markov approximations have the following properties.

\begin{lem}\label{lem:markov approx properties}
	Let  $\mu \in \MM_\sig(\Sigma)$ . The following hold:
	\begin{enumerate}[(1)]
		\item\label{it:markov approx ac margins} for every $k,n \in \N$, we have $\mu|_n \ll \mu^{(k)}|_n$,
		\item if $\mu \in \Ek_\sig(\Sigma)$, then  $\mu^{(k)} \in \Ek_\sig(\Sigma)$ for every $k \in \N$,
		\item $h(\mu||\mu^{(k)}) = h(\mu^{(k)}) - h(\mu)$ for every $k \in \N$,
		\item $\lim \limits_{k \to \infty} h(\mu||\mu^{(k)}) = 0$.
	\end{enumerate}
\end{lem}

\begin{proof}
	\begin{enumerate}[(1)]
		\item It is clear that $\mu|_n \ll \mu^{(k)}|_n$ for $n \leq k$ (the two distributions are equal in this case). For $n > k$, it follows from \eqref{eq:k-step Markov def} for $\nu = \mu^{(k)}$ that $\mu^{(k)}([\om_1, \ldots, \om_n]) = 0$ implies that either $\mu([\om_1, \ldots, \om_k]) = 0$ or $P_{\mu^{(k)}}(\om_{j+k}|\om_j, \ldots, \om_{j+k-1}) = 0$ for some $j \in \{1, \ldots, n-k\}$. As the latter case gives $\mu([\om_j, \ldots, \om_{j+k-1}]) = 0$ or $\mu([\om_j, \ldots, \om_{j+k}]) = 0$, we obtain $\mu([\om_1, \ldots, \om_n]) = 0$ in both cases, so  $\mu|_n \ll \mu^{(k)}|_n$.
		\item We will apply Lemma \ref{lem:markov ergodic} to $\mu^{(k)}$. Let $\om, \tau \in \mA^k$ be such that $\mu^{(k)}([\om]) > 0, \mu^{(k)}([\tau]) > 0$. Then also $\mu([\om]) > 0, \mu([\tau]) > 0$ and by ergodicity of $\mu$, there exists $u \in \Sigma_*$ with $\mu([u])>0$ such that $\om$ is a prefix of $u$ and $\tau$ is a suffix of $u$. By point \ref{it:markov approx ac margins} we have $\mu^{(k)}([u]) > 0$.
		\item This follows by combining Lemma \ref{lem:relative entropy markov formula} with \eqref{eq:markov entropy}, as $\mu^{(k)}|_{k+1} = \mu|_{k+1}$.
		\item This is \cite[Theorem 3.4]{GrayEntropyBook}.
	\end{enumerate}
\end{proof}

\begin{prop}\label{prop:separability}
	There exists a countable set $\Vk \subset \Ek_\sig(\Sigma)$ such that every $\nu \in \Vk$ is a $k$-step Markov measure for some $k \in \N$ and for every $\mu \in \Ek_\sig(\Sigma)$ and $\eps>0$ there exists $\nu \in \Vk$ such that $0 \leq h(\mu||\nu) < \eps$, and  $\mu|_n \ll \nu|_n$ for every $n \in \N$.
\end{prop}

\begin{proof}
	We define $\Vk$ to consist of all ergodic $k$-step Markov measures $\nu$ for which the transition kernel $P_\nu$ takes only rational values. This is clearly a countable set.  Fix $\mu \in \Ek_\sigma(\Sigma)$ and $\eps > 0$. By Lemma \ref{lem:markov approx properties} there exists $ k_0 $ such that for all $k \geq k_0$
	\begin{equation}\label{eq:markov approx bounds}
		0 \leq h(\mu||\mu^{(k)}) = h(\mu^{(k)}) - h(\mu) < \eps/2.
	\end{equation}
	Consider $\nu$ which is a $k$-step Markov measure for which the transition kernel $P_{\nu}$ has exactly the same zeros as $P_{\mu^{(k)}}$. As $\mu^{(k)}$ is ergodic by Lemma \ref{lem:markov approx properties}, every such $\nu$ is also ergodic (by Lemma \ref{lem:markov ergodic}) and hence stationary distribution $\nu|_k$ depends continuously on the transition kernel $P_{\nu}$ (by the uniqueness of the stationary distribution, as long as we preserve the zeros of the transition kernel, see \cite[Theorem 1.19]{WaltersErgodicBook}). Therefore $\nu|_{k+1}$ depends continuously on the transition kernel $P_\nu$ and hence we can find $\nu \in \Vk$ such that $\nu|_{k+1}$ and $\mu^{(k)}|_{k+1} = \mu|_{k+1}$ are arbitrarily close and have the same zeros. Consequently, by the formula in Lemma \ref{lem:relative entropy markov formula}, we can choose $\nu \in \Vk$ such that $|h(\mu||\nu) - h(\mu||\mu^{(k)})| < \eps/2$. Combining this with \eqref{eq:markov approx bounds} and noting that $\mu|_n \ll \mu^{(k)}|_n \ll \nu|_n$ for every $n \in \N$ (by Lemma \ref{lem:markov approx properties} and the fact that $P_{\mu^{(k)}}$ and $P_{\nu}$ have the same zeros) finishes the proof.
\end{proof}

It remains to connect the relative entropy with relative dimension. We do so for a class of metrics $\rho$ on $\Sigma$ satisfying the following assumptions:

\begin{enumerate}[(i)]
	\item\label{it: rho to psi} there exists a function $\psi : \Sigma^* \to (0,\infty)$ so that $\rho(\om, \tau) = \psi(\om \wedge \tau)$ for every $\om \neq \tau \in \Sigma$,
	\item\label{it: psi mono} there exists $\gamma \in (0,1)$ such that $\psi(\om|_{n+1}) \leq \gamma \psi(\om|_{n})$ for each $n \geq 1$ and $\om \in \Sigma$.

\end{enumerate}

\begin{lem}\label{lem: ball cylinder}
Let $\rho$ be a metric on $\Sigma$ satisfying \ref{it: rho to psi} and \ref{it: psi mono}. Fix $N \in \N$ so that $\gamma^N \leq 1/2$. Then for every $r > 0$ small enough the following holds: for every $\om \in \Sigma$ there exists $n \geq 1$ such that for
\begin{equation}\label{eq: ball cylinder} [\om|_{n+1}] \subset  B(\om,r) \subset [\om|_n]
\end{equation}
and
\begin{equation}\label{eq: 2r ball cylinder}
B(\om, 2r) \subset [\om|_{n-N}]
\end{equation}
where $B(\om, r)$ denotes the ball in the metric $\rho$. Moreover, $n$ can be chosen so that $n \leq \frac{\log r}{\log \gamma} + B$ for some constant $B$ (depending only on $\psi$).
\end{lem}

\begin{proof}
Take $0 < r < \min \{ \psi(i) : i \in \Ak \}$ and given $\om \in \Om$, let $n \geq 1$ to be the unique integer such that
\begin{equation}\label{eq: r n choice}  \psi(\om|_{n+1}) < r \leq \psi(\om|_n).
\end{equation}
If $\tau \in B(\om, r)$, then $\rho(\om, \tau) = \psi(\om \wedge \tau) < r \leq \psi(\om|_n)$, and hence, by \ref{it: psi mono},  $\om|_n$ is a prefix of $\om \wedge \tau$. This implies $\tau \in [\om|_n]$, so $B(\om, r) \subset [\om|_n]$. On the other hand, if $\tau \in [\om|_{n+1}]$, then $\om|_{n+1}$ is a prefix of $\om \wedge \tau$, so \ref{it: psi mono} gives $\rho(\om, \tau) = \psi(\om \wedge \tau) \leq \psi(\om|_{n+1}) < r$. Therefore $[\om|_{n+1}]\subset B(\om, r)$. This proves \eqref{eq: ball cylinder}. Set $A = \max\{ \psi(i) : i \in \Ak \}$ and note that \ref{it: psi mono} implies
\[ \psi(\om|_n) \leq A \gamma^{n-1} , \]
hence if \eqref{eq: r n choice} holds, then $r \leq A \gamma^{n-1}$, so $n \leq \frac{\log r}{\log \gamma} - \frac{\log A}{\log \gamma} + 1$. Finally, if \eqref{eq: r n choice} holds, then by \ref{it: psi mono}
\[ 2r \leq 2\psi(\om|_n) \leq 2\gamma^N \psi(\om|_{n-N}) \leq \psi(\om|_{n-N}), \]
since $N$ is chosen so that $\gamma^N \leq 1/2$. Then the same argument as before shows $B(\om, 2r) \subset [\om|_{n-N}]$, proving \eqref{eq: 2r ball cylinder}.
\end{proof}

Now we are ready to establish relative dimension separability of the set of ergodic measures on $\Sigma$ with respect to a large class of metrics.

\begin{prop}\label{prop: general ergodic rds}
Let $\rho$ be a metric on $\Sigma$ satisfying \ref{it: rho to psi} and \ref{it: psi mono}. Then the set $\Ek_\sigma(\Sigma)$ of ergodic shift-invariant probability measures on $\Sigma$ is relative dimension separable with respect to $\rho$.
\end{prop}

\begin{proof}
Let $\Vk \subset \Ek_\sig(\Sigma)$ be as in Proposition \ref{prop:separability}. Given $\mu \in \Ek_\sig(\Sigma)$ and $\eps > 0$, let $\nu \in \Vk$ be such that  $h(\mu||\nu) < \eps$. By Theorem \ref{thm:relative entropy ergodic thm}, for $\mu$-a.e. $\om \in \Sigma$ there exists $n_0 = n_0(\om, \eps)>0$ such that for all $n \geq n_0$
\begin{equation}\label{eq: meas comp cylinders}  2^{-\eps n} \nu([\om|_n])\leq \mu([\om|_n]) \leq  2^{\eps n} \nu([\om|_n]).
\end{equation}
Moreover, as $\nu$ is a $k$-step Markov measure we have
\[ K_{\nu}:= \max \left\{ \frac{\nu([\om|_n])}{\nu([\om|_{n+1}])} : \om \in \Sigma, n \geq k \text{ with } \nu([\om|_{n+1}]) > 0 \right\} < \infty,\]
as $K_{\nu} = \max \left\{ P_{\nu}(\om_{k+1} | \om_1, \ldots, \om_k)^{-1} : \om_1, \ldots, \om_{k+1} \in \Ak \text{ and } P_{\nu}(\om_{k+1} |  \om_1, \ldots, \om_k) \neq 0  \right\}$. Fix now $\om \in \supp (\mu)$ such that  \eqref{eq: meas comp cylinders} holds. Note that as $\om \in \supp (\mu)$, we have $\mu([\om|_n]) > 0$ for every $n \geq 1$ and hence \eqref{eq: meas comp cylinders} implies $\nu([\om|_n]) > 0$ for all $n \geq 1$. Let $A = \frac{1}{- \log \gamma}$. By Lemma \ref{lem: ball cylinder}, we have for all $r > 0$ small enough and $n$ satisfying $n_0 \leq n \leq A \log \frac{1}{r} + B$
\[
\mu(B(\om,r)) \leq \mu([\om|_n]) \leq  2^{\eps n} \nu([\om|_n]) \leq 2^{B\eps} K_{\nu}r^{-A\eps}\nu([\om|_{n+1}]) \leq 2^{B\eps}  K_{\nu}r^{-A\eps} \nu(B(\om,r)),
\]
and similarly
\[
\begin{split} \mu(B(\om, r)) & \geq \mu([\om|_{n+1}]) \geq  2^{-\eps (n+1)} \nu([\om|_{n+1}]) \geq (2^{\eps} K_{\nu})^{-1} 2^{-\eps n} \nu([\om|_n]) \\
&  \geq (2^{(B+1)\eps} K_{\nu})^{-1} r^{A\eps} \nu(B(\om,r)).
\end{split}\]
The last two inequalities give
\[ -A\eps \leq \liminf \limits_{r \to 0} \frac{\log \frac{\mu(B(\om,r))}{\nu(B(\om,r))}}{\log r}  \leq \limsup \limits_{r \to 0} \frac{\log \frac{\mu(B(\om,r))}{\nu(B(\om,r))}}{\log r} \leq A\eps.\]
As this holds for $\mu$-a.e. $\om \in \Sigma$, we have
\[ \dim(\mu||\nu, \rho) \leq A\eps. \]
Since $A$ is a constant depending only on $\rho$, we see that $\Ek_\sigma(\Sigma)$ is relative dimension separable with respect to $\rho$.
\end{proof}

If $\Fk$ is a conformal $C^{1+\theta}$ IFS, then setting $\psi(\om) = \|f'_\om\|$ one obtains Proposition \ref{prop: rel dim sep IFS} from Proposition \ref{prop: general ergodic rds}.

\subsection{Gibbs measures}
Let us recall definition of a Gibbs measure on a symbolic space $\Sigma = \Ak^\N$.

\begin{defn}\label{defn: Gibbs measure}
Let $\phi : \Sigma \to \R$ be a continuous function on $\Sigma$. A shift-invariant ergodic probability measure $\mu$ on $\Om$ is called a {\bf Gibbs measure of the potential} $\phi$ if there exists $P \in \R$ and $C\ge 1$ such that for every $\om \in \Sigma$ and $n \in \N$, holds the inequality
\begin{equation}\label{eq: Gibbs ineq} C^{-1} \leq \frac{\mu([\om|_n])}{\exp(-Pn + \sum \limits_{k=0}^{n-1} \phi(\sigma^k \om))} \leq C.
\end{equation}
\end{defn}
It is known that if $\phi$ is H\"older continuous, then there exists a unique Gibbs measure of $\phi$ (see \cite{BR08}). Here, H\"older continuity means H\"older continuity with respect to any metric of the form $\rho(\om, \tau) = \gamma^{|\om \wedge \tau|}$ on $\Sigma$ for $\gamma \in (0,1)$. We shall prove that the set of Gibbs measures is uniform relative dimension separable (recall Definition \ref{defn: urds}) with respect to metrics on $\Sigma$ satisfying \ref{it: rho to psi} and \ref{it: psi mono}.

\begin{prop}\label{prop: Gibbs urds}
Let $\rho$ be a metric on $\Sigma$ satisfying \ref{it: rho to psi} and \ref{it: psi mono}. Then the set $G_\sigma(\Sigma)$ consisting of all Gibbs measures corresponding to H\"older continuous potentials on $\Sigma$ is uniform relative dimension separable with respect to $\rho$. Moreover, each $\mu \in G_\sigma(\Sigma)$ is uniformly diametrically regular with respect to $\rho$.
\end{prop}

\begin{proof}
Let $\mu_\phi$ denote the Gibbs measure corresponding to a H\"older continuous potential $\phi: \Sigma \to \R$. It well known that the constant $P = P(\phi)$ for which $\mu_\phi$ satisfies \eqref{eq: Gibbs ineq} can be expressed via the following pressure formula
\[ P(\phi) = \lim \limits_{n \to \infty} \frac{1}{n} \log \sum \limits_{|u|=n} \exp (\sum \limits_{k=0}^{n-1} \phi(\sigma^k (u \om))) \]
for every $\om \in \Sigma$ (see e.g. \cite[Proof of Theorem 1.16]{BR08}). Am immediate consequence is the following: for a pair of  H\"older continuous potentials $\phi_1, \phi_2 : \Sigma \to \R$
\[ \|P(\phi_1) - P(\phi_2)\| \leq \|\phi_1 - \phi_2\|, \]
where $\|\cdot\|$ denotes the supremum norm on $\Sigma$ (see also \cite[Theorem 9.7.(iv)]{WaltersErgodicBook}). Combining this with \eqref{eq: Gibbs ineq} gives that for every $\phi_1, \phi_2$ there exists $M=M(\phi_1, \phi_2)$ such that for every $\om \in \Sigma$ and $n \in \N$
\begin{equation}\label{eq: Gibbs sup norm}
M^{-1}\exp(-2\|\phi_1 - \phi_2\|n) \leq \frac{\mu_{\phi_1}([\om|_{n-1}])}{\mu_{\phi_2}([\om|_{n}])}  \text{ and } \frac{\mu_{\phi_1}([\om|_n])}{\mu_{\phi_2}([\om|_{n-1}])} \leq M\exp(2\|\phi_1 - \phi_2\| n),
\end{equation}
To bound $\dim_u(\mu_{\phi_1}||\mu_{\phi_2}, \rho)$ we can use Lemma \ref{lem: ball cylinder}. Combined with \eqref{eq: Gibbs sup norm}, it gives that for  every $r>0$ small enough and every $\om \in \Sigma$ there exists $n \leq \frac{\log r}{\log \gamma} + B$ such that
\[\begin{split} \mu_{\phi_1}(B(x,r)) & \leq \mu_{\phi_1}([\om|_n]) \leq M\exp(2\|\phi_1 - \phi_2\|n)\mu_{\phi_2}([\om|_{n-1}]) \\
	& \leq M \exp(2\|\phi_1 - \phi_2\|B)\exp\left(2\|\phi_1 - \phi_2\|\frac{\log r}{\log \gamma}\right) \mu_{\phi_2}(B(x,r)) \\
	& = M' r^{-c\|\phi_1 - \phi_2\|}\mu_{\phi_2}(B(x,r)),
\end{split}\]
for a constant $M' = M'(\phi_1, \phi_2)$ and $c = c(\gamma)$ > 0. Repeating the above calculation using lower bounds, one obtains similarly
\[ \mu_{\phi_1}(B(x,r)) \geq M'^{-1}r^{c\|\phi_1 - \phi_2\|}\mu_{\phi_2}(B(x,r)) \]
(possibly increasing constants $M', c$). Those two bounds together give
\[ \dim_u(\mu_{\phi_1}||\mu_{\phi_2}, \rho) \leq c\|\phi_1 - \phi_2\|.\]
Therefore to prove that $G_\sigma(\Sigma)$ is uniform relative dimension separable it suffices to observe that the set of H\"older continuous functions $\phi : \Sigma \to \R$ is separable in the supremum norm. For example,  the following collection is a countable dense set: all functions $\phi : \Sigma \to \R$ for which there exists $n$ so that for every $\om \in \Sigma^*$ with $|\om| = n$ one has $\phi|_{[\om]} \equiv \mathrm{const} \in \Q$.

To prove that each $\mu \in G_{\sigma}(\Sigma)$ is uniformly diametrically regular, we invoke Lemma \ref{lem: ball cylinder} once more. Fix $N \in \N$ so that $\gamma^N \leq 1/2$  and take $r>0$ small enough so that for every $\om \in \Sigma$ there exists $n \in \N$ so that
\[ [\om|_{n+1}] \subset  B(\om,r) \subset [\om|_n] \text{ and } B(\om, 2r) \subset [\om|_{n-N}]. \]
Then by \eqref{eq: Gibbs ineq}
\[\begin{split} \mu(B(\om, 2r)) & \leq \mu([\om|_{n-N}]) \leq C \exp(-P(n-N) + \sum \limits_{k=0}^{n-N-1} \phi(\sigma^k \om)) \\
&  = C^2 \exp(P(N+1) - \sum \limits_{k=n-N}^n\phi(\sigma^{k}\om)) C^{-1} \exp(-P(n+1) + \sum \limits_{k=0}^{n} \phi(\sigma^k \om)) \\
& \leq C^2 \exp((N+1)(P+\|\phi\|)) \mu([\om|_{n+1}]) \\
& \leq M\mu(B(\om, r)),
\end{split}\]
for a constant $M$ depending only on $\rho$ and $\phi$. Given $\eps > 0$, this implies
\[ \mu(B(\om, 2r)) \leq r^{-\eps} \mu(B(\om,r)) \]
for all $r>0$ small enough to guarantee $r^{-\eps} \geq M$.
\end{proof}
\subsection{Exponential distance from the enemy}\label{subsec: EDE}

\begin{proof}[{\bf Proof of Proposition \ref{prop: ede holder}}]

	 First, note that we can assume that the attractor $\Lam$ of the IFS $\Fk$ is not a singleton (or equivalently $\Pi_\Fk$ is not constant), as otherwise Proposition \ref{prop: ede holder} holds trivially.
	
	We will make use of the bounded distortion property of $C^{1+\theta}$ conformal IFS: there exists a constant $C_D > 0$ such that 
	\begin{equation*}
	|f'_\om(x)| \leq C_D |f'_\om(y)| \text{ for every } x,y \in V \text{ and } \om \in \Sigma^*.
	\end{equation*}
	For the proof see e.g. \cite[Section 4.2]{MUbook}. An easy consequence (we use here that $\Lambda$ is not a singleton) is that there exist a constant $A>0$ such that
	\begin{equation}\label{eq: diam bounds}  A^{-1} \|f'_\om\| \leq \diam(f_\om(\Lambda)) \leq A \|f'_\om\| \text{ for all } \om \in \Sigma^*.
	\end{equation}
	
Assume that the EDE condition \eqref{eq:ede} is satisfied at $\om \in \Sigma$. Fix arbitrary $\alpha \in (0,1)$ and let $\eps > 0$ be such that $\alpha = \frac{1}{1+\eps}$. For $\tau \in \Sigma$ with $\tau \neq \om$, let $n = |\om \wedge \tau|$. Set $\gamma = \min \limits_{i \in \Ak} \inf \limits_{x \in V} \|f'_i(x)\|$ and note that by assumptions $\gamma > 0$. By \eqref{eq: diam bounds}, the EDE condition \eqref{eq:ede}  gives
	\[\begin{split} \rho_\Fk(\om,\tau) & = \|f'_{\om \wedge \tau}\| = \|f'_{\om|_n}\| \leq \gamma^{-1} \|f'_{\om|_{n+1}}\| \leq \frac{A}{\gamma} \diam(f_{\om|_{n+1}}(\Lambda)) =  \frac{A}{\gamma} \diam(\Pi_\Fk([\om|_{n+1}]))  \\
		&  \leq \frac{AC^{-\frac{1}{1+\eps}}}{\gamma} \mathrm{dist}\left(\Pi_\Fk(\om),\bigcup_{\substack{|v|=n+1\\ v\neq \om|_{n+1}}}\Pi_\Fk([v])\right)^{\frac{1}{1+\eps}} \\
		& \leq \frac{AC^{-\frac{1}{1+\eps}}}{\gamma}|\Pi_\Fk(\om) - \Pi_\Fk(\tau)|^\alpha,
	\end{split}\]
	hence \eqref{eq:ede} implies \eqref{eq:holder}.

In the other direction, assume that \eqref{eq:holder} holds. Given $\eps > 0$, set $\alpha = \frac{1}{1+\eps}$. For given $v \in \Sigma_*$ such that $|v| = n$ and $v \neq \om|_n$ take any $\tau \in [v]$. Then $\tau \wedge \om$ is a prefix of $\om|_n$,  so by \eqref{eq: diam bounds}
	
\[
\begin{split} |\Pi_\Fk(\om) - \Pi_\Fk(\tau)| & \geq C^{-\frac{1}{\alpha}} \rho_\Fk(\om,\tau)^{\frac{1}{\alpha}} = C^{-(1+\eps)} \|f'_{\om \wedge \tau}\|^{1+\eps} \\
	& \geq C^{-(1+\eps)} \|f'_{\om|_n}\|^{1+\eps} \geq (AC)^{-(1+\eps)}\diam(\Pi_\Fk([\om|_n]))^{1+\eps}.
\end{split}\]
As $\tau \in [v]$ is arbitrary we have
\[\mathrm{dist}\left(\Pi_\Fk(\om),\bigcup_{\substack{|v|=n\\ v\neq \om|_n}}\Pi_\Fk([v])\right)>\frac{ (AC)^{-(1+\eps)}}{2}\diam(\Pi_\Fk([\om|_n]))^{1+\eps}\]
	and hence \eqref{eq:holder} implies \eqref{eq:ede}.
\end{proof}

\section{Multifractal spectrum of self-similar measures}\label{sec: multifractal}

In this section we prove Theorem~\ref{thm: multi high dim main}. Let us begin with a general consequence of the nearly bi-Lipschitz property on the local dimensions of projections. Recall that for conformal iterated function systems, the nearly bi-Lipschitz property is equivalent to the EDE separation condition (Proposition \ref{prop: ede holder}).

\begin{lem}\label{lem: nearly biLip loc dim}
	Let $(X,\rho_X)$ and $(Y, d_{Y})$ be metric spaces.
	Let $\Pi: X \to Y$ be a Lipschtiz map and let $\mu$ be a finite Borel measure on $X$. If $\Pi$ is $\mu$-nearly bi-Lipschitz, then the following holds for $\mu$-a.e $x \in X$:
	\[ \ld(\Pi\nu, \Pi x) = \ld(\nu, x) \text{ and }  \ud(\Pi\nu, \Pi x) = \ud(\nu, x) \text{ for every finite Borel measure } \nu \text{ on } X. \]
\end{lem}

\begin{proof}
As we assume that $\Pi$ is Lipschitz, inequalities
\begin{equation}\label{eq: lipschitz loc dim} \ld(\Pi\nu, \Pi x) \leq \ld(\nu, x) \text{ and }  \ud(\Pi\nu, \Pi x) \leq \ud(\nu, x)
\end{equation}
hold for every $x \in X$ and every finite Borel measure $\nu$ on $X$. It therefore remains to prove the opposite inequalities for $\mu$-a.e. $x \in X$.
Let $x \in X$ and $\alpha \in (0,1)$ be such that there exists $C=C(x,\alpha)$ so that
\begin{equation}\label{eq: alpha holder inverse}
	\rho_X(x,y) \leq C \rho_Y(\Pi(x), \Pi(y))^\alpha \text{ for every } y \in X.
\end{equation} 
Then for all $r > 0$
\[ \Pi^{-1}(B(\Pi(x), r)) \subset B(x, Cr^\alpha), \]
so for any finite Borel measure $\nu$ on $X$
\[ \Pi  \nu (B(\Pi(x), r)) \leq \nu (B(x, Cr^\alpha)).\]
Therefore, for $0 < r < 1$
\[ \frac{\log \Pi\nu(B(\Pi(x),r))}{\log r} \geq \frac{\log \nu(B(x, Cr^\alpha))}{\log r} = \frac{\log Cr^\alpha}{\log r} \cdot \frac{\log \nu(B(x, Cr^\alpha))}{\log Cr^\alpha}. \]
Taking $\liminf \limits_{r \to 0}$ and $\limsup \limits_{r \to 0}$ yields
\begin{equation}\label{eq: loc dim lower bound}
	 \ld(\Pi \nu, \Pi (x)) \geq \alpha\ld(\nu, x) \text{ and } \ud(\Pi \nu, \Pi (x)) \geq \alpha\ud(\nu, x).
\end{equation}
If $\Pi$ is $\mu$-nearly bi-Lipschitz, then for $\mu$-a.e. $x \in X$, inequality \eqref{eq: alpha holder inverse} holds for every $\alpha \in (0,1)$, and hence \eqref{eq: loc dim lower bound} gives
\[ \ld(\Pi\nu, \Pi x) \geq \ld(\nu, x) \text{ and }  \ud(\Pi\nu, \Pi x) \geq \ud(\nu, x) \]
for $\mu$-a.e. $x$ and every $\nu$. Combined with \eqref{eq: lipschitz loc dim}, this finishes the proof.
\end{proof}

Now we can turn to proving Theorem~\ref{thm: multi high dim main}. We will use the following notation: for a self-similar measure $\nu$ on $\R^d$ corresponding to the IFS $\Fk=\{ x \mapsto \lam_i O_i x + t_i\}_{i \in \Ak}$ with $\lam_i \in (0,1)$, $d \times d$ orthogonal matrices $O_i$ and $t_i \in \R$ and a probability vector $p = (p_i)_{i\in\Ak}$, given $\alpha \in \R$ we set
\[ T^*(\alpha) := \inf_{q  \in \R}(\alpha q+T(q)), \]
where $T(q)$ is the unique real solution of $\sum_{i\in\Ak}p_i^q\lambda_i^{T(q)}=1$. First, we establish the lower bound.

\begin{lem}\label{lem:forupper}
	Let $\Fk=\{f_i(x)=\lambda_iO_ix+t_i\}_{i\in\Ak}$ be an IFS consisting of similarities of $\R^d$, with similarity dimension $s_0 = s(\Fk)$, let $p=(p_i)_{i\in\Ak}$ be a probability vector and let $\nu$ be the corresponding self-similar measure. If $(p_i)_{i\in\Ak}\neq(\lambda_i^{s_0})_{i\in\Ak}$ and $s_0<d$ then for every $\alpha\in\left[\frac{\sum_{i}\lambda_i^{s_0}\log p_i}{\sum_{i}\lambda_i^{s_0}\log\lambda_i},\max \limits_{i \in \Ak}\frac{\log p_i}{\log\lambda_i}\right]$
	\[
	\dim_H \left( \left\{x:d(\nu, x)=\alpha\right\}\right)\leq T^*(\alpha) = \inf_{q\leq0}(\alpha q+T(q))\]
\end{lem}

The proof of the lemma is standard, but we include it for completeness.

\begin{proof}
If $\alpha\in\left[\frac{\sum_{i}\lambda_i^{s_0}\log p_i}{\sum_{i}\lambda_i^{s_0}\log\lambda_i},\max \limits_{i \in \Ak}\frac{\log p_i}{\log\lambda_i}\right)$, then there exists unique $q = q_\alpha \leq 0$ such that
\begin{equation}\label{eq: alpha T der}
T'(q) = -\alpha,
\end{equation}
and moreover
\begin{equation}\label{eq: alpha q rel}
T'(q)=-\frac{\sum_ip_i^q|\lambda_i|^{T(q)}\log p_i}{\sum_ip_i^q|\lambda_i|^{T(q)}\log|\lambda_i|}\text{ and }q\alpha+T(q)=-qT'(q)+T(q)=T^*(\alpha),
\end{equation}
see for example \cite[Chapter~11]{Falconertechniques} or \cite[Chapter~5]{BSS}. Consequently
	\begin{equation}\label{eq: symbolic legendre negative} T^*(\alpha) = \inf_{q\leq0}(\alpha q+T(q)).
	\end{equation}
	By continuity, \eqref{eq: symbolic legendre negative} extends to $\alpha = \max \limits_{i \in \Ak}\frac{\log p_i}{\log\lambda_i}$.	It therefore suffices to show that for every $\alpha\in\left[\frac{\sum_{i}\lambda_i^{s_0}\log p_i}{\sum_{i}\lambda_i^{s_0}\log\lambda_i},\max \limits_{i \in \Ak} \frac{\log p_i}{\log\lambda_i}\right]$ and $q \leq 0$ inequality
	\[ \dim_H\left(\left\{x:d(\nu, x)=\alpha\right\} \right)\leq \alpha q + T(q) \]
	holds.

	Let $\mu = p^\N$ be the Bernoulli measure corresponding to $p$. By definition,
	$$
	\left\{x:d (\nu, x)=\alpha\right\}=\bigcap_{p=1}^\infty\bigcup_{N=1}^\infty\bigcap_{n=N}^{\infty}\{x:2^{-n(\alpha+1/p)}\leq\nu(B(x,2^{-n}))\leq 2^{-n(\alpha-1/p)}\}.
	$$
	Hence, it is enough to show that for every $p\geq1$, $N\geq1$, $q\leq0$ and $\varepsilon>0$
	\begin{equation}\label{eq:enoughmulti}
		\dim_H \left( \bigcap_{n=N}^{\infty}\{x:2^{-n(\alpha+1/p)}\leq\nu(B(x,2^{-n}))\leq 2^{-n(\alpha-1/p)}\}\right)\leq q(\alpha-1/p)+T(q)+\varepsilon.
	\end{equation}
	Let $X_{N,p}=\bigcap_{n=N}^{\infty}\{x:2^{-n(\alpha+1/p)}\leq\nu(B(x,2^{-n}))\leq 2^{-n(\alpha-1/p)}\}$ and $\mathcal{B}^N=\{B(x,2^{-n}):n\geq N\text{ and }x\in X_{N,p}\}$.	Then by Besicovitch's covering theorem (see for example \cite[Theorem~B.3.2]{BSS}), there exists $Q=Q(d)$ such that for every $i=1,\ldots, Q$ there exists a countable family $\Bk_i^N\subseteq\Bk^N$ such that for every distinct $B,B'\in\Bk_i^N$, $B\cap B'=\emptyset$ and $X_{N,p}\subseteq\bigcup_{i=1}^Q\bigcup_{B\in\Bk^N_i}B$. Clearly, $X_{N,p}\subseteq X_{N+1,p}$, and so $\bigcup_{i=1}^Q\Bk_i^M$ is a cover for $X_{N,p}$ for every $M\geq N$.
	Hence,
	\[\begin{split}
		\mathcal{H}_{2^{-M}}^{q(\alpha-1/p)+T(q)+\varepsilon}(X_{N,p})&\leq\sum_{i=1}^Q\sum_{B\in\Bk_i^M}|B|^{q(\alpha-1/p)+T(q)+\varepsilon}\leq\sum_{i=1}^Q\sum_{B\in\Bk_i^M}\nu(B)^q|B|^{T(q)+\varepsilon}=\star.
	\end{split}\]
	For every $B\in\Bk^M$ there exists $u\in\Sigma$ such that $\Pi(u)$ is the center of $B$, furthermore, there exists a minimal $n=n(B)\geq1$ such that $\Pi([u|_n])\subseteq B$. Since $q\leq 0$ we get
	\[\begin{split}
		\star&\leq\sum_{i=1}^Q\sum_{B\in\Bk_i^M}\mu([u(B)|_{n(B)}])^q|B|^{T(q)+\varepsilon}\lesssim 2^{-M\varepsilon}\sum_{i=1}^Q\sum_{B\in\Bk_i^M}\mu([u(B)|_{n(B)}])^q\lambda_{u(B)|_{n(B)}}^{T(q)}\\
		&=2^{-M\varepsilon}\sum_{i=1}^Q\sum_{B\in\Bk^M_i}p_{u(B)|_{n(B)}}^q\lambda_{u(B)|_{n(B)}}^{T(q)}.
	\end{split}\]
	Since $\Bk_i^M$ is formed by disjoint balls, the cylinders $\{[u(B)|_{n(B)}]\}_{B\in\Bk_i^M}$ are disjoint too, and by the definition of $T(q)$, we get $\mathcal{H}_{2^{-M}}^{q(\alpha-1/p)+T(q)+\varepsilon}(X_{N,p})\lesssim 2^{-M\varepsilon}$, which implies \eqref{eq:enoughmulti}.
\end{proof}

\begin{proof}[{\bf Proof of Theorem~\ref{thm: multi high dim main}}]
 Our goal is to prove that for almost every $t$, equality
	\begin{equation}\label{eq:multifractal}
		\dim_H \left( \left\{x:d(\nu_{t,p}, x)=\alpha\right\} \right)=\inf_{q\in\R}(\alpha q+T(q))
	\end{equation}
	holds simultaneously for all $\alpha\in\left[\frac{\sum_{i}\lambda_i^{s_0}\log p_i}{\sum_{i}\lambda_i^{s_0}\log\lambda_i},\max \limits_{i \in \Ak}\frac{\log p_i}{\log\lambda_i}\right]$. This will finish the proof in the case $d>1$. For $d=1$ we can additionally apply \cite[Theorem~1.2, Remark~7.3]{BarralFeng}, showing that for almost every $t$,  \eqref{eq:multifractal} holds for $\alpha \in \left[ \min \limits_{i \in \Ak} \frac{\log p_i}{\log\lambda_i}, \frac{\sum_{i}\lambda_i^{s_0}\log p_i}{\sum_{i}\lambda_i^{s_0}\log\lambda_i}\right]$. The upper bound $	\dim_H \left( \left\{x:d(\nu_{t,p}, x)=\alpha\right\} \right) \leq \inf \limits_{q\in\R}(\alpha q+T(q))$ follows by Lemma~\ref{lem:forupper}.

	We shall prove the lower bound for all $\alpha\in\left[\min \limits_{i \in \Ak} \frac{\log p_i}{\log\lambda_i},\max \limits_{i \in \Ak}\frac{\log p_i}{\log\lambda_i}\right]$. Let us first introduce some notation. Let $\Sigma = \Ak^N$ be the symbolic space. Given $\alpha\in\left(\min \limits_{i \in \Ak} \frac{\log p_i}{\log\lambda_i},\max \limits_{i \in \Ak}\frac{\log p_i}{\log\lambda_i}\right)$, let $q_\alpha \in \R$ be such that \eqref{eq: alpha T der} and \eqref{eq: alpha q rel} hold with $q = q_\alpha$ (recall e.g.  \cite[Chapter~5]{BSS}). Further, for a given probability vector $p = (p_i)_{i \in \Ak}$, let $\nu^p = p^{\otimes \N}$ be the Bernoulli measure on $\Sigma$ corresponding to $p$ and let $\mu_{\alpha, p}$ be the Bernoulli measure on $\Sigma$ corresponding to the probability vector $(p_i^{q_\alpha}|\lambda_i|^{T(q_\alpha)})_{i\in\Ak}$. Let $\Pi_t : \Sigma \to \R^d$ be the natural projection map corresponding to the IFS $\Fk_t$. Note that with this notation $\nu_{t,p} = \Pi_t \nu^p$. Endow $\Sigma$ with the adapted metric $\rho = \rho_{\Fk_t}$ defined in \eqref{eq: adapted metric} and corresponding to the IFS $\Fk_t$ (note that $\rho_{\Fk_t}$ does not depend on $t$; in fact it depends only on contractions $\lam_i, i \in \Ak$, which we treat as fixed). A direct computation (see e.g. \cite[Lemma 5.1.3]{BSS}) shows that (with respect to metric $\rho$ on $\Sigma$)
	\begin{equation}\label{eq: nu t p dim}
	 d(\nu^{p}, \om) = \alpha \text{ for } \mu_{\alpha, p}\text{-a.e. } \om \in \Sigma
	\end{equation}
	and, by \eqref{eq: dim symbolic}, \eqref{eq: alpha T der} and \eqref{eq: alpha q rel}
	\begin{equation}\label{eq: mu alpha p dim}
		\dim_H \mu_{\alpha, p}  =\frac{h(\mu_{\alpha, p})}{\chi(\mu_{\alpha, p})}=-q_\alpha\frac{\sum_ip_i^{q_\alpha}|\lambda_i|^{T(q_\alpha)}\log p_i}{\sum_ip_i^{q_\alpha}|\lambda_i|^{T(q_\alpha)}\log|\lambda_i|}+T(q_\alpha)  =q_\alpha\alpha+T(q_\alpha)=T^*(\alpha).
	\end{equation}
	By Example \ref{ex:translate}, the assumptions of Theorem \ref{thm: multi high dim main} are sufficient to establish the transversality condition for the family $\Fk_t$ with $t$ as the parameter. We can therefore apply Theorem \ref{thm: general proj} (or, more directly, Theorem \ref{thm: main trans IFS}) and, as we assume $s_0 = s(\Fk_t) < d$, we can conclude that for $\eta$-a.e. $t \in (\R^d)^{\Ak}$, simultaneously for every $\alpha$ and $p$, the natural projection $\Pi_t$ is $\mu_{\alpha, p}$-nearly bi-Lipschitz, and (recalling \eqref{eq: mu alpha p dim}) equality
	\begin{equation}\label{eq: pi t mu alpha p dim}
	\dim_H \Pi_t\mu_{\alpha, p} = T^*(\alpha)
	\end{equation} holds. It therefore follows from Lemma \ref{lem: nearly biLip loc dim} and \eqref{eq: nu t p dim} that for $\eta$-a.e. $\lam$ and every $\alpha,p$
	\[ d(\nu_{t,p}, \Pi_t(\om)) = d(\Pi_t\nu^{p}, \Pi_t(\om)) = d(\nu^{p}, \om) = \alpha \text{ for } \mu_{\alpha, p}\text{-a.e.} \om \in \Sigma.\]
	This implies (as the above shows $\Pi_t\mu_{\alpha, p}\left( \{x:d ( \nu_{t,p}, x)=\alpha\} \right) = 1$)
	\[ \dim_H \left( \{x:d ( \nu_{t,p}, x)=\alpha\} \right) \geq \hdim \Pi_t \mu_{\alpha, p} \]
	and so by \eqref{eq: pi t mu alpha p dim}
	\begin{equation}\label{eq: multifractal lower bound} \dim_H \left( \{x:d ( \nu_{t,p}, x)=\alpha\} \right) \geq  T^*(\alpha)
	\end{equation}
	This establishes \eqref{eq:multifractal} for  $\alpha\in\left[\frac{\sum_{i}\lambda_i^{s_0}\log p_i}{\sum_{i}\lambda_i^{s_0}\log\lambda_i},\max \limits_{i \in \Ak}\frac{\log p_i}{\log\lambda_i}\right)$.
	
	To complete the proof, if $\alpha=\max \limits_{i \in \Ak}\frac{\log p_i}{\log|\lambda_i|}$ then in order to obtain the lower bound (the upper bound in this case follows from Lemma \ref{lem:forupper}), one can let $\Ak_{\max}=\left\{i\in\Ak:\frac{\log p_i}{\log|\lambda_i|}=\alpha\right\}$ and repeat the argument above with $\mu_{\alpha, p}$ being the uniformly distributed Bernoulli measure on $\Ak_{\max}^\N$. This establishes \eqref{eq:multifractal} for $\alpha\in\left[\frac{\sum_{i}\lambda_i^{s_0}\log p_i}{\sum_{i}\lambda_i^{s_0}\log\lambda_i},\max \limits_{i \in \Ak}\frac{\log p_i}{\log\lambda_i}\right]$. For $\alpha=\min \limits_{i \in \Ak} \frac{\log p_i}{\log|\lambda_i|}$, one can similarly repeat the argument to obtain the lower bound \eqref{eq: multifractal lower bound}.
	
\end{proof}

Note that it was crucial for the above proof that for almost every translation parameter $t$, the $\mu_{\alpha, p}$-nearly bi-Lipschitz property (i.e. EDE) holds simultaneously for all $\alpha$ (and $p$), as follows from the universal projection theorem (Theorem \ref{thm: general proj}). If it was established only for fixed $\alpha$, one would obtain equality \eqref{eq:multifractal} only for a single level-set rather then on the full spectrum $\left(\frac{\sum_{i}\lambda_i^{s_0}\log p_i}{\sum_{i}\lambda_i^{s_0}\log\lambda_i},\max \limits_{i \in \Ak}\frac{\log p_i}{\log\lambda_i}\right]$.

\bibliographystyle{alpha}
\bibliography{selfsimilar_bib}

\end{document}